\documentclass[a4paper,11pt,reqno]{amsart} 
\usepackage{amssymb, amsfonts, stmaryrd, color,fancyhdr, amsmath,amsthm,amsbsy,amsfonts, titletoc,  mathrsfs, mathabx, verbatim,extarrows, wasysym, euscript, enumerate, latexsym, mathtools, epsfig, subfigure, mathabx, graphicx, color, wrapfig, hyperref,graphicx,float, bm, setspace}
\usepackage{CJKutf8}
\usepackage[margin=1in]{geometry}

\usepackage[pagewise,displaymath,mathlines]{lineno}
\newcommand*\patchAmsMathEnvironmentForLineno[1]{%
   \expandafter\let\csname old#1\expandafter\endcsname\csname #1\endcsname
   \expandafter\let\csname oldend#1\expandafter\endcsname\csname end#1\endcsname
   \renewenvironment{#1}%
      {\linenomath\csname old#1\endcsname}%
      {\csname oldend#1\endcsname\endlinenomath}}%
\newcommand*\patchBothAmsMathEnvironmentsForLineno[1]{%
   \patchAmsMathEnvironmentForLineno{#1}%
   \patchAmsMathEnvironmentForLineno{#1*}}%
\AtBeginDocument{%
\patchBothAmsMathEnvironmentsForLineno{equation}%
\patchBothAmsMathEnvironmentsForLineno{align}%
\patchBothAmsMathEnvironmentsForLineno{flalign}%
\patchBothAmsMathEnvironmentsForLineno{alignat}%
\patchBothAmsMathEnvironmentsForLineno{gather}%
\patchBothAmsMathEnvironmentsForLineno{multline}%
}

\numberwithin{equation}{section}
\theoremstyle{plain}
\DeclareMathAlphabet{\mathpzc}{OT1}{pzc}{m}{it}

\def\md{\mathrm d}

\def\mrF{\mathrm F}

\def\mrL{{\mathrm L}}

\def\mrd{{\mathrm d }}

\def\Htwoin{ H^2_{\rm in}}

\newcommand{\beq}{\begin{equation}}
\newcommand{\eeq}{\end{equation}}
\newcommand{\beqs}{\begin{equation*}}
\newcommand{\eeqs}{\end{equation*}}
\newcommand{\h}{\hspace}
\newcommand{\normmm}[1]{{\left\vert\kern-0.15ex\left\vert\kern-0.15ex\left\vert #1 
    \right\vert\kern-0.15ex\right\vert\kern-0.15ex\right\vert}}

\newcommand{\varep}{\varepsilon}

\newcommand{\p}{\partial}

\newtheorem{thm}{Theorem}[section]
\newtheorem{lemma}[thm]{Lemma}
\newtheorem{cor}[thm]{Corollary}

\newtheorem{prop}[thm]{Proposition}
\newtheorem{defn}[thm]{Definition}

\begin{document}
\title{\textbf{Volume Preserving Willmore Flow in a  generalized Cahn-Hilliard Flow} }
\author{
Yuan Chen 
}
\address{Division of Mathematics, SSE, Chinese University of Hong Kong(Shenzhen), Shenzhen, China}
\curraddr{}
\email{chenyuan@cuhk.edu.cn}
\thanks{The author was supported  by NSF Youth Grant of China  \#12301262.}


\date{\today}

\dedicatory{}


\begin{abstract}
We investigate the mass-preserving $L^2$-gradient flow associated with a generalized Cahn--Hilliard equation. Our focus is on the sharp interface regime, where the interface width parameter $\varepsilon > 0$ is small. For well-prepared initial data, we rigorously prove that, as $\varepsilon \to 0$, solutions of the diffuse-interface model converge to the \emph{volume-preserving Willmore flow} in arbitrary spatial dimensions $n \geq 2$. The proof incorporates matched asymptotic expansions and energy estimates to establish convergence of the order parameter away from the interface, alongside precise motion law derivation for the limiting interface. This result extends the analysis of Fei and Liu~\cite{fei2021phase} from two-dimensional settings to general $n$-dimensional domains, and it applies to a broad class of symmetric double-well potentials beyond the classical quartic form. Our work thus provides a general PDE framework linking higher-order phase-field models to volume-preserving curvature flows in the sharp interface limit.
\end{abstract}
\maketitle


\section{Introduction}

\noindent


In this article, we consider a generalized Cahn-Hilliard model. {On a periodic domain $\Omega \subset \mathbb{R}^N$, the free energy is expressed in terms of the phase variable $u$.} Precisely, the free energy functional, denoted as $\mathcal F$, in the context of generalized Cahn-Hilliard, see \cite{du2004phase,du2009energetic,colli2012phase}, is defined as: 
\begin{equation}\label{FCH}
\mathcal F(u):=\int_\Omega \frac{1}{2\varep} \left(\varep\Delta u-\frac{1}{\varep}W'(u)\right)^2dx. 
\end{equation}
Here,  $\varepsilon \ll 1$  is a small parameter controlling the thickness of the sharp(single layer) interface $\Gamma$ between different phases; $ W(u)$  is a  double-well potential, with local minima at  $b_-$ (pure solvent phase) and  $b_+$  (pure oil phase), it is symmetric with respect to  $u=(b_++b_-)/2$ and satisfies 
\begin{equation}\label{def-CH-energy}
    b_-<b_+, \quad W(b_-)=0=W(b_+), \quad \hbox{and} \quad W''(b_\pm)>0. 
\end{equation}
 Without loss of generality, we take $b_\pm=\pm1$ in this article. This energy describes the Canham-Helfrich bending interfacial energy  and helps enforce smooth but sharp transitions between phases; see \cite{du2004phase} for instance. Let $H$ be the mean curvature of the interface, then the Canham-Helfrich bending energy, \cite{canham1970minimum,helfrich1973elastic}, is defined by
 \begin{equation}
     \mathcal E = \int H^2 d\mu_\Gamma, 
 \end{equation}
 subject to prescribed volume and surface area. The Canham-Helfrich model serves as a fundamental tool in both theoretical and applied studies of membrane biophysics, offering a framework to predict and analyze the complex behaviors of cellular membranes, see \cite{seifert1997configurations}. We would like to mention the generalized Cahn-Hilliard model is comparable to the functionalized Cahn-Hilliard model, where the interfaces form bilayers; see, for example, \cite{promislow2009pem, dai2013geometric, kraitzman2015overview, chen2023curve}. 



\subsection{Mass preserving gradient flow} In many cases, phase field models lead to nonlinear dynamics governed by gradient flows of an energy functional. Gradient flows describe the evolution of a system in the direction of steepest descent of the energy functional, thus capturing how interfaces form, evolve, and stabilize. 

Gradient flow are described by chemical potentials, the variational derivative of the free energy functional $\mathcal{F}(u)$ with respect to the phase field $u$. Precisely,  the chemical potential of the energy $\mathcal F$, denoted by $\mathrm F=\mrF(u)$,  is defined as: 
\beq
\label{eq-FCH-L2-p}
\mathrm F(u):= \frac{\delta \mathcal F}{\delta u}=\varep^{-1}\left(\varep\Delta -\frac{1}{\varep}W''(u)\right)\left(\varep\Delta u-\frac{1}{\varep}W'(u)\right).
\eeq
In this article, we 
consider a mass-preserving $L^2$ gradient flow given by: 
 \begin{equation}\label{eq-FCH-L2}
\varep\p_t u_\varep=-\Pi_0  \mathrm F(u_\varep).
\end{equation}
where $\mathrm{F}(u_\varepsilon)$ represents the variational derivative (the chemical potential) given in \eqref{eq-FCH-L2-p} and $\Pi_0$ is the zero-mass projection that ensures mass conservation. For a function  $f$  that is integrable on a bounded domain $\Omega$, the zero-mass projection operator $\Pi_0$ is defined as:
\begin{equation}
 \Pi_0 f:=f-\overline f, 
\end{equation}
where $\overline{f}$ is the average of f over the domain $\Omega$, given by:
\beq
\label{def-massfunc}
\overline f:=\frac{1}{|\Omega|}\int_\Omega f d x.
\eeq
Here, $|\Omega|$ denotes the Lebesgue measure (or volume) of the domain $\Omega$, and the integral represents the mean value of the function $f$ over the domain. {The zero-mass projection $\Pi_0$ removes the mean value of the function, resulting in a function whose integral over the domain is zero.  In fact, 
\begin{equation}
    \frac{d}{dt} \int_\Omega u\, dx = -\frac{1}{\varep}\int_\Omega \Pi_0 \mathrm F(u_\varep)\, dx=0.  
\end{equation}
This implies the following mass condition holds for all time: 
\begin{equation}\label{mass-condition} 
\int_\Omega u_\varepsilon \, dx =  M_0,
\end{equation}
where $M_0$ is a constant that relates to the initial  mass of  molecules in the system.}
The generalized Cahn-Hilliard (gCH) flow, given by equations \eqref{eq-FCH-L2-p}-\eqref{eq-FCH-L2}, can be expressed as a second-order system. The rewritten system is:
\begin{equation}\label{eq-2nd-order-system}
    \begin{aligned}
        \varep^3 \p_t u_\varep &=- \Pi_0  \left[\left(\varep^2\Delta - W''(u_\varep)\right)v_\varep \right]; \\
        \varep v_\varep & = \varep^2\Delta u_\varep-W'(u_\varep).
    \end{aligned}
\end{equation}

\subsection{Main results.} 
One of the central challenges in studying the dynamics of phase field models is understanding the convergence behavior of solutions as the parameter $\varepsilon$ becomes small. This parameter controls the thickness of the transition layer between phases. 
The first rigorous justification of the dynamic limit of the Cahn-Hilliard equation was provided by Alikakos et al. in \cite{alikakos1994convergence}, establishing a mathematical foundation for the use of the Cahn-Hilliard model as an approximation for Hele-Shaw or Mullins-Sekerka flows. This is based on the formal result established in \cite{pego1989front}.  Building on this foundational work, the primary objective of this article is to rigorously demonstrate that, as $\varepsilon$ becomes small, the solutions of the mass-preserving generalized Cahn-Hilliard (gCH) model  converge to the solutions of a corresponding sharp-interface model. 

In the first part of the main result, we establish the existence of order $k$-approximate smooth solutions to the mass-conserved system defined by equations \eqref{eq-2nd-order-system}-\eqref{mass-condition}. For these approximate solutions $(u_a, v_a)$, we introduce the remainder terms $(\mathcal{R}_1, \mathcal{R}_2)$, which accounts for the error between the exact solution and the approximate solution. These remainder terms are expressed in the following system:
\begin{equation}\label{def-remainder-R12}
   \left\{ \begin{aligned}
         \varep^3 \p_t u_a &=- \Pi_0  \left[\left(\varep^2\Delta - W''(u_a)\right)v_a \right] +\mathcal R_1; \\
       \varep v_a & = \varep^2\Delta u_a-W'(u_a)+\mathcal R_2.
    \end{aligned}\right. 
\end{equation}
An order $k$-approximate solution to the system \eqref{eq-2nd-order-system}-\eqref{mass-condition} is defined as follows:
\begin{defn}[$k$-approximate solution]\label{def-k-approximate sol} Let  $k \geq 1$  be a positive integer, $T>0$ be a given positive constant. A pair $(u_a, v_a)$ is called a k-approximate solution to the system \eqref{eq-2nd-order-system}-\eqref{mass-condition} on the domain $\Omega \times [0, T)$ if the following conditions hold:
    \begin{enumerate}
        \item The system holds approximately up to order k, with remainder terms $\mathcal{R}_1$ and $\mathcal{R}_2$ bounded by $\varep^{k+1}$ up to  a multiplying constant $C_{\mathcal{R}}$, that is, 
        \begin{equation}
            \|\mathcal R_1\|_{L^\infty(\Omega\times [0, T)])} +\|\mathcal R_2\|_{L^\infty(\Omega\times [0, T)])} \leq C_{\mathcal R}\varep^{k+1}. 
        \end{equation}
        \item The mass conservation condition holds approximately up to order $k$, meaning:
        \begin{equation}
            \left|\int_\Omega u_adx- M_0\right| \leq C_{\mathcal R}\varep^{k+1}. 
        \end{equation}
    \end{enumerate}
    Alternatively, we also say that $u_a$ is a $k$-approximate solution to the fourth order flow \eqref{eq-FCH-L2}. 
     The constant $C_{\mathcal{R}}$ is a bounding constant depending on system parameters, the order $k$, and the time interval $T$, ensuring uniform control over the accuracy of the approximate solution.
\end{defn}

Approximate solutions are constructed through asymptotic analysis around an interface, starting with the expansion of the solution and the distance function in both the fast (inner) and slow (outer) regions, expressed in terms of the small parameter $\varepsilon$. This process reduces the governing equation into a hierarchy of equations at different orders, and finding an approximate solution involves solving these equations. The solvability of these equations determines the dynamics of the interface. At leading order, the interface, denoted by $\Gamma_0 = \Gamma_0(t)$, evolves according to a Willmore flow  subject to a volume constraint. This evolution is described below in terms of the distance function $d_0$. Let $V_0$ denote the normal velocity, defined by:
\begin{equation}
\begin{aligned}
    V_0:=\left(-\Delta^2 d_0+(\Delta d_0+\nabla d_0\cdot \nabla)D_0\right)\bigg|_{\Gamma_0}
    \end{aligned} 
\end{equation}
where 
\begin{equation}
        D_0 := \nabla \Delta d_0 \cdot \nabla d_0 +\frac{1}{2} (\Delta d_0)^2. 
    \end{equation}
Then the leading order dynamics, i.e. evolution of $\Gamma_0$, is given by 
\begin{equation}\label{def-dynamics-Gamma0}
    \p_t d_0 = V_0+\frac{2\sigma_2}{m_1^2}, 
\end{equation} 
where  $\sigma_2$ is the Lagrange multiplier determined by the leading order surface enclosed volume constraint 
\begin{equation}\label{Area-constraint-leading-order}
    |\mathcal V_0|=\frac12 (|\Omega|-M_0). 
\end{equation}
 The  enclosed region of the surface $\Gamma_0$ is defined as the set where $d_0<0$ or $u<0$. Hereafter, the flow \eqref{def-dynamics-Gamma0}-\eqref{Area-constraint-leading-order} is referred to as  the Willmore flow  subject to a volume constraint. Well-posedness of  Willmore flows with various constraints has been studied by many mathematicians; we refer the interested reader to  \cite{simonett2001willmore,kuwert2002gradient,abels2016local,lecrone2019surface}. The relation between the classical Willmore flow in terms of curvature and distance function can be found, for instance, in \cite{droske2004level,bellettini2014lecture,bretin2015phase}. Particularly, the normal velocity $V_0$ is given by the variational derivative of the Canham-Helfrich energy in \eqref{def-CH-energy}. 
  Similarly to  the established work in Cahn-Hilliard case \cite{alikakos1994convergence,fei2021phase}, as long as the leading order geometric flow keeps smooth,   $k$-approximate solutions are close to an actual solution.  This motivates us to define the compatible data as below: 
 \begin{defn}[Compatible Data]\label{def-compatible-data}
    We call $(\Gamma_0, T)\in \mathbb R^N\times \mathbb R$  a  compatible data  if $\Gamma_0=\Gamma_0(t)$, solving the flow \eqref{def-dynamics-Gamma0}-\eqref{Area-constraint-leading-order}, is smoothly embedded in $\Omega$ for all $t\in[0,T]$.
\end{defn} 
 In order to deal with the nonlinear term in high dimensions. An extra assumption is  put on the double-well potential $W$. Particularly, we consider   $W=W(\phi)$ is a polynomial of degree $2\beta$ for some integer $\beta\geq 2$ and the coefficient of $\phi^{2\beta}$ is positive, that is, for some $c_0>0$, $W$ takes the form
\begin{equation}\label{assumption-W}
    W(\phi)=c_0\phi^{2\beta}+\tilde W(\phi),
\end{equation} 
where $\tilde W$ is some polynomial function   of $\phi$.   It's still quite general which also involves a typical double well potential.  
Now we are in the position  to state our main theorem below. 
\begin{thm}\label{thm-main}
Let $k\geq 1$ be a given positive integer.  Suppose that $(\Gamma_0, T)$ is compatible data, then  there is a  $k$-approximate solution $(u_a, v_a)$ to the system \eqref{mass-condition}-\eqref{eq-2nd-order-system} on $\Omega \times [0,T)$. Moreover, if $k> \max\{ (N+3),   10\}$, then there exists a positive constant $C_0$, independent of $\varep\leq 1$,  such that for all the initial data $u_\varep(\cdot,0)$ satisfying 
        \begin{equation}
            \|u_\varep(\cdot, 0)-u_a(\cdot, 0)\|_{L^2(\Omega)}^2\leq C_0 \min\{\varep^{2N},\varep^{10}\}.
        \end{equation}
        the $k$-approximate solution $u_a$ is close  to  the exact solution of the system \eqref{eq-2nd-order-system}-\eqref{mass-condition} on $\Omega\times [0, T)$ subject to initial data $u_\varep(\cdot, 0)$. More precisely, there exists $C>0$ which might depend on $T$ and $\Omega$ but is independent of $\varep$  such that 
\begin{equation}
        \|u_\varep(\cdot, t)-u_a(\cdot,t)\|_{L^2(\Omega)}^2\leq C \min\{\varep^{2N},\varep^{10}\}, \quad \forall t\in [0,T). 
    \end{equation} 
\end{thm}
 

\medskip
\noindent
The proof builds on the approach developed in \cite{alikakos1994convergence,fei2021phase} and involves three main components. First, we perform an asymptotic expansion in local coordinates to construct $k$-th order approximate solutions. Second, we establish spectral estimates for the linearized operator, with particular attention to proving coercivity. Third, we derive energy estimates to control the nonlinear terms and close the argument. 
This article also contains several extensions of previous work. Notably, the model incorporates an additional mass constraint, leading to an algebraic coupling that complicates the construction of approximate solutions. Moreover, the analysis is generalized to spatial dimensions $n \geq 2$, including cases beyond two and three dimensions, and the framework accommodates a broader class of symmetric double-well potentials. For related results in the second-order setting, we refer the reader to \cite{chen2011mass}, which rigorously justifies the convergence of a mass-preserving Allen--Cahn model to volume-preserving mean curvature flow. Our work can be viewed as a higher-order analogue of this result, extending such convergence theory to fourth-order phase-field gradient flows. 

The convergence result obtained in this work is stated in terms of the $L^2$-norm of the order parameter. Higher-order regularity estimates for the difference between the diffuse and sharp interface solutions can, in principle, be deduced from the well-posedness theory together with suitable interpolation inequalities. For the well-posedness of this type of phase-field flow, we refer the interested reader to Wu~\cite{wu2022well} and the references therein. We also highlight another recent development concerning the generalized Cahn--Hilliard equation, due to Liu~\cite{liu2024solutions}, where the authors constructed multi-layer solutions in one spatial dimension and in three dimensions with radial symmetry. These results further illustrate the rich variety of patterns and dynamics that can arise in higher-order phase-field models.

The remainder of this article is arranged as the following. In section \ref{sec-preliminaries}, we set up some general geometric framework and notations; some asymptotic expansion of geometric quantities and operators are also introduced. At Section \ref{sec-approximate-solutions}, we introduce the asymptotic expansion method briefly  and construct the $k$-approximate solutions. In Section \ref{sec-main-theorem}, we state our main theorem in this article and the proof is based on the coercivity of linearized operator established in \ref{sec-linear-coercivity}.    Lastly, in Appendix A we give some technical  lemmas used in the article. Appendix B is devoted to  solving  the order-by-order system from asymptotic expansion, and exploring the relation between the surface enclosed volume and  background state. 

\subsection{Notation}\label{sec-Notation} We present some general notation. 
\begin{enumerate}
\item  The symbol $C$ generically denotes a positive constant whose value depends only on the domain $\Omega$, and 
geometric quantities of the surface $\Gamma$. In particular its value is independent of $\varep$ so long as it is sufficiently small. 
The value of $C$  may vary line to line without remark.  In addition, $A\lesssim B$ indicates that quantity $A$ is less than quantity $B$ up 
to a multiplicative constant $C$ as above. 
The expression $f=O(a)$  indicates the existence of a constant $C$, as above,
and a norm $|\cdot|$ for which
\begin{equation*}
| f| \leq C |a|.
\end{equation*}
\item The quantity $\nu$ is a positive number, independent of $\varep$, that denote an exponential decay rate. It may vary from line to line.
\item If a function space $X(\Omega)$ is comprised of functions defined on the whole spatial domain $\Omega$, we will drop the symbol $\Omega$.
\end{enumerate} 

\section{Geometric setup}\label{sec-preliminaries}

\subsection{Local coordinates}
Let $\{\Gamma(t)\}_{t\geq 0}$ be a family of  compact smooth co-dimension one surfaces embedded in $\mathbb R^n$ and $\mathbf X(\cdot;t): U\to \Gamma(t)$ is a local parameterization of it with $U$ being an open set in $\mathbb R^{N-1}$. 
For a sufficiently small $\ell>0$, the $\ell$-tubular neighborhood is well-defined and shall be denoted as $\Gamma^\ell(t)$. This neighborhood consists of points within a distance less than $\ell$ from the surface, measured along the normal direction at each point. Let $s=(s_1,\cdots, s_{N-1})$ be the local coordinates on the surface $\Gamma(t)$ and $\mathbf n$ is the unit outer normal, then for each $x\in \Gamma^\ell(t)$ there exists unique $(s,r)$ such that 
\begin{equation}\label{eq-local-coordinates}
    x= \mathbf X(r,s) = \mathbf X(s)+ r\mathbf n(s). 
\end{equation} 
Here $r$ is the signed distance from the surface along the normal direction, $r \mathbf n(s)$ is the displacement along the normal vector $\mathbf n(s)$ from $\Gamma$ by a distance r.  $(s,r)$ is called the local coordinates of $\Gamma^\ell(t)$. As a convention of this article, we represent a geometric quantity $A(r,s)$ at $(r,s)$ with $r=0$ by $A(s)$, that is, $A(s)=A(0,s)$. 

The tubular neighborhood $\Gamma^\ell$ can be interpreted as a collection (or union) of surfaces located at different distances r from the base surface $\Gamma$. Specifically, for each value of $r \in [-\ell, \ell]$, the set of points $\mathbf X(s, r)$ defines a surface, which we can denote as $\Gamma^{\ell, r}(t)$, located at a distance $r$ from the original surface $\Gamma$. Thus, the tubular neighborhood can be written as:
\begin{equation}
\Gamma^\ell(t) = \bigcup_{|r| \leq \ell} \Gamma^{\ell, r}(t).
\end{equation}
Here we note that $\Gamma^{\ell, 0}=\Gamma$ is the base surface.  
For simplicity of notation the parameter $t$ is usually omitted without ambiguity.  
\subsection{Geometric quantities}
Under the local coordinates, the metric is given by 
\begin{equation}\label{eq-metric}
    (g_{ij})_{N\times N}=\left(
        \begin{array}{cc}
            (\p_{s_i}\mathbf X\cdot \p_{s_j}\mathbf X)_{i,j\leq N-1} & 0 \\
            0 & 1
        \end{array}\right).
\end{equation}
This matrix captures the geometric structure of the tubular neighborhood, where $\partial_{s_i}\mathbf X \cdot \partial_{s_j} \mathbf X$  represents the components of the metric along the surface $\Gamma^{\ell, r}$, derived from the surface parameterization; the lower right component  represents the contribution of the radial distance $r$, which comes directly from the displacement along the normal direction. The inverse of this metric is denoted by  $(g^{ij})$, following standard notation in differential geometry. The determinant of the metric tensor, denoted by  $g = \det(g_{ij})$, is assumed to be nondegenerate. This nondegeneracy condition is expressed as:
\begin{equation}\label{bdd-metric-determinant} 
    |\ln g(r,s)|\leq C  
\end{equation}
 for some constant $C$ depending only on $\ell$ and $\Gamma$. This condition ensures that the metric remains well-behaved (i.e., invertible) throughout the tubular neighborhood, which is essential for applying the coordinate transformations and computing geometric quantities. Each of the level surfaces $\Gamma^{\ell, r}$, where $r$ is a constant, has a surface measure $dS_r$ given by:
\begin{equation}
    dS_r = \sqrt{g(r,s)} ds. 
\end{equation}
where we recall that $g(r, s)$ is the determinant of the metric tensor on the surface $\Gamma^{\ell,r}$, $ds$ is the surface element in the local coordinates on $\Gamma$. For the base surface $r = 0$, the surface measure is denoted simply as $dS = dS_0$. 

The second fundamental form of the level surfaces $\Gamma^{\ell,r}$ is defined by:  
\begin{equation}\label{def-hij} 
    h_{ij}(r,s) = \mathbf n(r,s)\cdot \p_{s_is_j} \mathbf X(r,s), \qquad i,j=1,2,\cdots N-1.   
\end{equation} 
The principal curvatures, denoted as $\{\kappa_i\}_{i=1}^{N-1}$, are the eigenvalues of the matrix representation of the second fundamental form relative to the first fundamental form (the metric). Since $\mathbf n$ is perpendicular to the tangent vectors $\p_{s_i}\mathbf X$, the second fundamental form can also be represented as:  $h_{ij} = - \p_{s_j}\mathbf n\cdot \p_{s_i} \mathbf X$. Using the first and second fundamental forms, the vector $\p_{s_i} \mathbf n$ can be represented by \begin{equation}\p_{s_i} \mathbf n = -\sum_{l,k=1}^{N-1}h_{il}\, g^{lk}\, \p_{s_k} \mathbf X,  \end{equation}
and 
\begin{equation}
    \p_{s_i} \mathbf n\cdot \p_{s_j} \mathbf n =\sum_{l,k=1}^{N-1} g^{lk}h_{il}h_{jk}.
\end{equation}
  The first fundamental form can be represented in terms of the base surface geometric quantities as: 
\begin{equation}\label{1st-fundamental form-expansion} 
    g_{ij}(r,s)=\left<\p_{s_i}\mathbf X, \p_{s_j}\mathbf X\right> = g_{ij}(s) -2r h_{ij}(s) +r^2\sum_{l,k=1}^{N-1} g^{lk}(s)h_{il}(s)h_{jk}(s).  
\end{equation} 

{The Jacobian $J(r, s)$ measures the change in the surface area with respect to the local parameterization determined by $(r,s)$, and is given by 
\begin{equation}
    J(r,s):= \sqrt{g(r,s)}. 
\end{equation}
Under appropriate parameterization of the base surface, we may assume that the Jacobian of the base surface is one, i.e. $J(0,s)=1$.} Using the Jacobian derivative implies
\begin{equation}\label{eq-derivative-Jacobian-1}
    \p_r \ln J(r, s) = -\frac12 \sum_{i,j=1}^{N-1} g^{ij}(r, s)  \p_r g_{ij}(r,s). 
\end{equation}
From \eqref{1st-fundamental form-expansion} the formula of the derivative becomes 
\begin{equation}
    \p_r \ln J(r, s) = -\sum_{i,j=1}^{N-1}g^{ij}(r,s) \left(h_{ij}(s) -r\sum_{l,k=1}^{N-1}g^{lk}(s)h_{il}(s)h_{jk}(s)\right). 
\end{equation}
where $H=H(s)$ is the mean curvature of the base interface defined by 
\begin{equation}\label{def-mean-curvature}
    H(s) = -\sum_{i,j=1}^{N-1} g^{ij}(s) h_{ij}(s). 
\end{equation}
From the definition of the inverse matrix, we have
\begin{equation}
    \p_r g^{ij}(r,s) = -\sum_{k,l=1}^{N-1}g^{il}(r,s) g^{kj}(r,s) \p_r g_{lk}(r,s).  
\end{equation}
And taking one more derivative on both sides of \eqref{eq-derivative-Jacobian-1} and using the identity above implies 
\begin{equation}
    \p_r^2 \ln \sqrt{J(r,s)} = \frac12 \sum_{i,j,k,l=1}^{N-1}  g^{il}(r,s) g^{kj}(r,s) \p_r g_{lk}(r,s)\p_rg_{ij}(r,s)-\frac12 \sum_{i,j=1}^{N-1} g^{ij}(r, s)  \p_r^2 g_{ij}(r,s)
\end{equation} 
Evaluating the relation at $r=0$ yields 
\begin{equation}
    \p_r^2 \ln J(r,s)\big|_{r=0}=H_2(s) := g^{il}(s) g^{kj}(s) h_{lk}(s)h_{ij}(s) . 
\end{equation}
Similarly, one can calculate any other order of the derivatives in terms of the first and second fundamental forms. Particularly, these would imply
\begin{equation}
    \p_r \ln J(r,s) = H(s) + r H_2(s) +r^2 H_3(s)+O(r^4),
\end{equation}
where $H(s)$ is the mean curvature. 
Similarly,  the Jacobian $J(r, s)$ can be expanded in terms of the first and second fundamental forms of the base surface $\Gamma$ and we write  
\begin{equation}\label{expansion-Jacobian} 
    J(r,s)=  1+r{H}(s)+r^2 \kappa_2(s)+r^3\kappa_3(s) +O(r^4),  
\end{equation} 
where  $\kappa_2(s), \kappa_3$ are higher-order terms determined by $H,H_2,H_3$. 
On a bounded  domain, the expansion of the Jacobian implies: if $\Gamma$ is a smooth surface with bounded curvatures, then 
\begin{equation}\label{est-Jacobian}
    |J-1|\leq C \varep |z|, \qquad |J-1-\varep z{H}|\leq C\varep^2 z^2.
\end{equation}
From the relation \eqref{bdd-metric-determinant}, the Jacobian is regular as 
\begin{equation}\label{lower-bdd-Jacobian} 
    J(r, s) \geq C, \qquad \hbox{on $\Gamma^\ell$}
\end{equation}
for some positive constant $C$. Using the Jacobian, the integral of a  function $f$, supported in $\Gamma^\ell$ can be rewritten as 
\begin{equation}
    \int_{\Omega} f(x)dx = \int_{-\ell/\varep}^{\ell/\varep} \int_{\Gamma}f(x(s,z)) \varep J(\varep z,s)dsdz. 
\end{equation}

\subsection{Operators under local coordinates}   Under the new local coordinates $(s,r)$,  the  Lalpace-Beltrami operator on each level surface $\Gamma^{\ell,r}$ is
 \begin{equation}\label{def-Laplace-Beltrami-Gamma}
     \Delta_\Gamma = \sum_{i,j=1}^{N-1} \frac{1}{J} \p_{s_i} \left(J g^{ij} \p_{s_j}\right). 
 \end{equation}
In the tubular neighborhood $\Gamma^\ell$, the Euclidean Laplace operator $\Delta$, which operates on scalar functions f(x) defined in the surrounding space, can be written in terms of the Laplace-Beltrami operator $\Delta_\Gamma$ on the surface level sets and a term involving derivatives in the normal direction r. Specifically, the Euclidean Laplace operator in the local coordinates $(s, r)$ is:
 \begin{equation}\label{eq-Laplace-operator}
     \Delta = \Delta_\Gamma +\frac{1}{J} \p_r \left(J \p_r\right). 
 \end{equation} 
The Laplace operator in the tubular neighborhood can be expanded in terms of the radial coordinate $r$.  Precisely, the expansion of the Laplacian is:
\begin{equation}\label{eq-Laplace-expansion-0}
    \varep^2\Delta = \varep^2\Delta_\Gamma +\varep^2 \p_r^2+\varep^2 \p_r\left( \ln J\right) \p_r.  
\end{equation}
The third term $\partial_r \left( \ln J \right) \partial_r$ accounts for the change in the metric due to the geometry of the surface. Note $\nabla r=\mathbf n$ and $\p_r\ln J =\Delta r$. In the expansion of the metric determinant, we have:
\begin{equation}\label{expansion--metric-derivative}
    \begin{aligned}
    &\p_r\ln J= \Delta r ={H}(s) +H_2(s)r+H_3(s) r^2+\mathcal O(r^3);\\
    & \p_r^2 \ln J = \nabla \Delta r\cdot \nabla r = H_2(s)+2H_3(s)r+\mathcal O(r^2),  
    \end{aligned} 
\end{equation}
where $H_2(s)$ and $H_3(s)$ are coefficients that depend on the higher-order curvature terms.
Using  \eqref{eq-Laplace-operator}-\eqref{eq-Laplace-expansion-0}  the Laplacian in local coordinates is expanded as:
 \begin{equation}\label{expansion-Laplace}
        \varep^2 \Delta = \varep^2 \Delta_\Gamma+\p_z^2 +\varep\Big( {H}(s)  +\varep zH_2(s)\Big)\p_z+\varep^3  \mathrm D_z. 
    \end{equation}
where $\mathrm D_z$,  accounts for additional corrections due to the geometry of the surface, is defined as   
\begin{equation}
\mathrm D_z:=\varep^{-2}\left(\p_r\ln\sqrt{g}-{H}(s)-\varep zH_2(s)\right)\p_z;
\end{equation} 
Here we recall that $  {H}(s)$  represents the mean curvature of the surface  $\Gamma $; $H_2(s)$ represents higher-order curvature corrections, see \eqref{expansion--metric-derivative}. 
For signed distance $r=\sum_k \varep^k d_k $, we expand
\begin{equation}
    \nabla \Delta r\cdot \nabla r +\frac12 |\Delta r|^2 =\sum_{k} \varep^k D_k, 
\end{equation} 
where  $D_k$ is the $k$-th term in the expansion, given by:
\begin{equation}\label{def-Dk}
\begin{aligned}
    D_0&=\nabla \Delta d_0\cdot \nabla d_0 +\frac12 (\Delta d_0)^2;\\
    D_k&= \sum_{i=0}^k  \left(\nabla \Delta d_i\cdot \nabla d_{k-i} +\frac12 \Delta d_i \Delta d_{k-i}\right).
    \end{aligned} 
\end{equation}
We point out the following relation:  
\begin{equation}
    D_0 +\varep D_1= \frac12 H^2(s) + H_2(s)+r\left(HH_2+2H_3\right)+O(r^2). 
\end{equation}

\section{Approximate solutions}\label{sec-approximate-solutions}

In this section we construct  $k$-approximate solutions to the second order systems \eqref{eq-2nd-order-system} as defined in \eqref{def-k-approximate sol}. Introducing the Lagrange multiplier $\sigma_\varep$ from mass constraint,  the system \eqref{eq-2nd-order-system} reduces to 
\begin{equation}\label{eq-2nd-order-system-sigma}
    \begin{aligned}
       \varep^3  \p_t u_\varep &=-  \left(\varep^2\Delta -W''(u_\varep)\right)v_\varep +\sigma_\varep; \\
        \varep v_\varep & = \varep^2\Delta u_\varep-W'(u_\varep).
    \end{aligned}
\end{equation}
Subject to periodic boundary condition, the Lagrange multiplier $\sigma_\varep$  satisfies 
\begin{equation}\label{def-sigma}
    \sigma_\varep= -\frac{1}{|\Omega|}\int_\Omega  W''(u_\varep)v_\varep dx. 
\end{equation}
This ensures that the mass constraint is satisfied as time varies. Here we recall that  $|\Omega|$ denotes the volume of the domain $\Omega$.

\subsection{Inner expansion}
Let $\Gamma=\Gamma_\varep(t)$ be an embedded smooth moving surface, it has a tubular neighborhood with width $\ell$, denoted as $\Gamma^\ell$, on which every point  can be represented uniquely in terms of the local coordinates such as $x=\mathbf X_\varep +d_\varep\mathbf n_\varep$, where $d_\varep$ is the signed distance of $x$ to $\Gamma$. The signed distance function  $d_\varepsilon$  satisfies: $|\nabla d_\varepsilon| = 1$. This condition implies that $ d_\varepsilon$  is the distance from the point  $x$  to the surface  $\Gamma$, and its gradient has unit magnitude, which is typical for distance functions.  

Let  $z = \frac{d_\varepsilon(x,t)}{\varepsilon}$  be the rescaled distance variable, a common approach, see \cite{alikakos1994convergence,chen2011mass,fei2021phase} for instance, is to expand the solution in powers of  $\varepsilon$, which reflects how the solution behaves near the surface  $\Gamma$.  Precisely, the solutions in the local region $\Gamma^\ell$ expand as 
\begin{equation}\label{Expansion-u-v}
\begin{aligned}
    u_\varep(x,t) &= \tilde u_\varep \left(\frac{d_\varep(x,t)}{\varep}, x,t\right). \quad \tilde u_\varep(z, x,t) = \sum_{i=0}^\infty \varep^i u_i(z,x,t);\\
    v_\varep(x,t) &= \tilde v_\varep \left(\frac{d_\varep(x,t)}{\varep}, x,t\right). \quad \tilde v_\varep(z, x,t) = \sum_{i=0}^\infty \varep^i v_i(z,x,t).
    \end{aligned}
\end{equation}
This form of expansion expresses the solution as a function of both the fast variable  $z = \frac{d_\varepsilon}{\varepsilon}$, which captures variations normal to the surface  $\Gamma $, and the slow variables  $x$  and  $t$, which capture variations along the surface and in time. Using the chain rule and the fact that  $|\nabla d_\varepsilon| = 1$, the Laplace and time derivative operators are expanded as: 
\begin{equation}
\begin{aligned}
   &\varep^2  \Delta= \p_z^2 +2\varep \nabla_x d_\varep \cdot \nabla_x \p_z  +\varep \Delta_x d_\varep \p_z+\varep^2 \Delta_x; \\ 
   & \varep^2 \p_t u_\varep =\varep^2 \p_t \tilde u_\varep + \varep \p_z \tilde u_\varep \p_t d_\varep. 
   \end{aligned}
\end{equation}
$\nabla, \Delta$ denote the total gradient and total Laplacian with respect to the spatial variable  $x$  when acting on  $u_\varepsilon$. They account for the full spatial dependence of the solution  $u_\varepsilon(x,t)$  in both the normal and tangential directions of the surface  $\Gamma_\ell$.  $\nabla_x, \Delta_x $ represent the partial derivatives with respect to the  variable  $x$  when acting on  $\tilde{u}_\varep(z, x, t)$.  The governing equations for  $\tilde{u}_\varepsilon $ and  $\tilde{v}_\varepsilon$  are given as: 
\begin{equation}\label{eq-u-varep}
    \left\{\begin{aligned}
       &  \p_z^2 \tilde u_\varep-W'(\tilde u_\varep) = \varep\left(\tilde v_\varep-\Delta_x d_\varep \p_z \tilde u_\varep -2\nabla_x d_\varep \cdot \nabla_x \p_z \tilde u_\varep\right) -\varep^2 \Delta_x \tilde u_\varep; \\
       &\p_z^2\tilde v_\varep - W''(\tilde u_\varep) \tilde v_\varep =\sigma_\varep -\varep \Big(\Delta_x d_\varep  \p_z\tilde v_\varep+2\nabla d_\varep \cdot \nabla_x  \p_z\tilde v_\varep \Big)-\varep^2 \p_td_\varep\p_z \tilde u_\varep\\
       &\h{4cm} -\varep^3 \p_t\tilde u_\varep-\varep^2 \Delta_x \tilde v_\varep .
    \end{aligned}\right. 
\end{equation} 
These two equations hold merely for $(z,x,t)\in \mathbb R\times \Gamma^\ell\times (0,T)$ for some $T>0$. The second equation is modified to account for an extra term, which arises from the need to match the condition  $|\nabla d_\varepsilon| = 1$  (i.e., the distance function must satisfy this normalization), see \cite{fei2021phase}.  This results in the addition of the term  $\varepsilon^2 E_\varepsilon(d_\varepsilon - \varepsilon z) {\phi_0'}(z)$, where $E_\varep=E_\varep(x,t)$ is free to choose. 
The modified equation is written as 
\begin{equation}\label{eq-v-varep}
    \begin{aligned}
        &\p_z^2\tilde v_\varep - W''(\tilde u_\varep) \tilde v_\varep =\sigma_\varep  -\varep \Big(\Delta_x d_\varep   \p_z\tilde v_\varep+2\nabla d_\varep \cdot \nabla_x  \p_z\tilde v_\varep \Big)-\varep^2 \p_td_\varep\p_z \tilde u_\varep -\varep^3 \p_t\tilde u_\varep\\
       &\h{3.2cm}-\varep^2 \Delta_x \tilde v_\varep +\varep^2E_\varep(d_\varep-\varep z){\phi_0'}. 
    \end{aligned}
\end{equation}
To construct a $k$-approximate solution to \eqref{eq-2nd-order-system}-\eqref{def-sigma}, we find a solution  $(u_\varep, v_\varep, \sigma_\varep, d_\varep, E_\varep)$ to the system, \eqref{def-sigma}, \eqref{eq-u-varep}$_1$-\eqref{eq-v-varep}, formally up to an order of $\varep$. For this purpose, we shall introduce  expansions for the distance function
\begin{equation}\label{Expansion-d}
    d_\varep=\sum_{i=0}^\infty \varep^i d_i(x,t);
\end{equation}
for the Lagrange multiplier
\begin{equation}\label{Expansion-sigma}
    \sigma_\varep = \sum_{i=0}^\infty \varep^i\sigma_i;
\end{equation} 
and for the extra justification term 
\begin{equation}\label{Expansion-E}
    E_\varep(x,t)=\sum_{i=0}^\infty \varep^iE_i(x,t). 
\end{equation} These expansions, together with the expansions for  $\tilde{u}_\varepsilon $ and  $\tilde{v}_\varepsilon$, \eqref{Expansion-u-v},  in powers of  $\varepsilon$, allow the system of equations \eqref{eq-u-varep}, \eqref{eq-v-varep} to be separated into terms of different orders of  $\varepsilon$. 
These are given in the appendix B, or Section \ref{Appendix-B}. 
The leading-order equation gives a simplified form that describes the behavior of the solution in the limit (i.e.,  $\varepsilon \to 0 $); The next order in  $\varepsilon$  provides corrections to this leading-order behavior.

To ensure the solvability of the system at each order of  $\varepsilon$, compatibility conditions must be satisfied. This shall be  introduced in the following text.   
\subsection{Compatibility condition}
Let $\phi_0=\phi_0(z)$ be the heteroclinic solution to the ordinary differential equation(ODE): 
    \begin{equation}\label{eq-phi0}
         \phi''(z)-W'(\phi)=0, \qquad \lim_{z\to \pm\infty}\phi(z)=\pm 1. 
    \end{equation} 
    The existence and exponential convergence of the solution \( \phi_{0} \) are ensured by the following lemma.  
A more precise statement for the singular problem on the real axis \( \mathbb{R} \) and for general \( W \) can be found in \cite{chen2011mass} (see also Lemma~4.1 in \cite{alikakos1994convergence} for the special case of a double-well potential). Here, we quote the result and provide a self-contained proof for the exponential decay.  
   {\begin{lemma}\label{lem-phi0}
        Let $W$ be a general symmetric double-well potential satisfying the assumptions \eqref{def-CH-energy}, then there is an odd solution  $\phi_0$, unique subject to translations, to the ODE system which is increasing, odd and converging exponentially to $\pm 1$ as $z\to \pm \infty$. 
    \end{lemma}
    \begin{proof}
        The existence of a monotone solution 
\[
\phi_{0} \in H^{1}(\mathbb{R}) \hookrightarrow C^{0}(\mathbb{R})
\]
satisfying the boundary condition in \eqref{eq-phi0} can be established using variational methods; see, for instance, \cite{berestycki1983nonlinear} and references therein.  
Higher regularity of \( \phi_{0} \) then follows from the classical theory of ordinary differential equations.  

In what follows, we present the proof of the exponential convergence of \( \phi_{0} \) by adopting a classical approach (see, for example, \cite{peletier1983uniqueness}).  
Without loss of generality, we restrict our attention to the case \( z > 0 \).  
Introducing the new variable,  
        \begin{equation}
            w(z)= \phi_0'(z) /\phi_0(z)>0. 
        \end{equation} 
        Here, we note that \( w > 0 \) on \( \mathbb{R}_{+} \) since \( \phi_{0} > 0 \) for \( z > 0 \) and \( \phi_{0}' > 0 \).  
We first observe that \( w \) is uniformly bounded for \( z > 0 \) away from \( z = 0 \).  
In this case, both \( \phi_{0} \) and \( \phi_{0}' \) converge to zero as \( z \to \infty \), and L'Hospital's rule applies, yielding  
\begin{equation}
    \lim_{z \to \infty} \frac{(\phi_{0}')^{2}}{\phi_{0}^{2}} 
    = \lim_{z \to \infty} \frac{\phi_{0}''}{\phi_{0}} 
    = \lim_{z \to \infty} \frac{W'(\phi_{0})}{\phi_{0}} 
    = \lim_{z \to \infty} W''(\phi_{0}) 
    = W''(1).
\end{equation}
Using an integrating factor, we deduce that
\begin{equation}
    \phi_{0}(z) \sim 1 + O\!\left(e^{-\sqrt{W''(1)}\,z}\right),
    \qquad \text{as } z \to \infty.
\end{equation}  It remains to establish the uniform boundedness of $w$. First,  using the equation of $\phi_0$,  the new variable $w$ satisfies 
        \begin{equation}\label{eq-w} 
            w' =-w^2 +\frac{W'(\phi_0)}{\phi_0}. 
        \end{equation}
        Since $\phi_0$ converges to $1$ as $z\to\infty$, and $W'(\phi_0)/\phi_0\to W''(1)>0$,  there exists a $z_0>0$ such that 
        \begin{equation}
            0<\frac{W'(\phi_0)}{\phi_0} < 1+W''(1), \qquad \forall z\geq z_0. 
        \end{equation}
        Suppose that 
\[
w(z) > \sqrt{1 + W''(1)}
\]
for some \( z \geq z_0 \). Then the equation for \( w \) implies that \( w'(z) < 0 \), and hence \( w(z) \) decreases until
\[
w(z) < \sqrt{\frac{W'(\phi_0)}{\phi_0}}<\sqrt{1+W''(1)}.
\]
Therefore, if \( w(z) \) is initially greater than \( \sqrt{1 + W''(1)} \), it will decrease initially and will never exceed this value for larger \( z \).   Hence
        \begin{equation}
            w(z) \leq \max\{w(z_0), 1+W''(1)\}, \qquad \forall z\geq z_0. 
        \end{equation}
   This completes the proof of the exponential decay of \( \phi_{0} \).  
Moreover, combined with the uniform boundedness of \( w \), this also implies the exponential decay of \( \phi_{0}' \).  
From the equation satisfied by \( \phi_{0} \), it follows that \( \phi_{0}'' \) also decays exponentially.  
The decay of higher-order derivatives can then be established by mathematical induction, obtained by differentiating the equation repeatedly.  
Hence, the proof is complete.
    \end{proof}}
    The solution  $\phi_0$  describes a “kink” or transition between two stable states of the system (or phases).   The linearized operator around the solution  $\phi_0$  is defined as: 
    \begin{equation}\label{def-L0}
        \mrL_0 = -\p_z^2+W''(\phi_0).  
    \end{equation}
    The constant $ m_1$ are defined as 
    \begin{equation}\label{def-m01}
    \begin{aligned} 
        m_1:=\|\phi_0'\|_{L^2(\mathbb R)},
        \end{aligned} 
    \end{equation}
    where 	 $m_1$  is the  $L^2(\mathbb R)$ -norm of  $\phi_0'$, the derivative of the heteroclinic solution. These constants are important in normalizing eigenfunctions and describing the behavior of perturbations around  $\phi_0$.
    
    \begin{lemma}\label{lem-L0}
        The spectrum of $\mrL_0$ is real and uniformly positive, except for the one point spectra: $\lambda_0=0$. Moreover, 
        \begin{equation}
            \mrL_0 \phi_0'=0, \quad \mrL_0\phi_0'' =-W'''(\phi_0) |\phi_0'|^2,\quad \mrL_0(z\phi_0')=-2\phi_0''. 
        \end{equation}
        The kernel of $\mrL_0$ is spanned by $\psi_1=\phi_0'/m_1$, where $m_1$ is the normalizing constant defined in \eqref{def-m01}. 
    \end{lemma}
   These relations in the Lemma describe how the operator  $\mathrm L_0$  acts on derivatives of the homoclinic solution and how higher-order nonlinear terms involving the potential  $W(\phi_0)$  relate to the structure of the equation. Particularly, these imply
   \begin{equation}\label{relation-L-W}
        W'''(\phi_0)|\phi_0'|^2 = - \mrL_0\phi_0''.
    \end{equation}
   Alternatively this relation above combined with the last relation in Lemma \ref{lem-L0} implies 
    \begin{equation}\label{identities-K2-integration}
    \begin{aligned} 
       & \int_{\mathbb R} W'''(\phi_0) z(\phi_0')^3 dz=2\int_{\mathbb R} |\phi_0''|^2 dz.
       \end{aligned} 
    \end{equation}
Regarding the inhomogeneous equation, such as $\mathrm L_0w=f$, the Fredholm alternative provides a solvability condition. Specifically, for the solution to exist, the right-hand side  $f(z)$  must be orthogonal to the kernel of  $\mathrm L_0$. {More precise statement for the singular problem (on the real axis $\mathbb R$) and general $W$  can be found see \cite{chen2011mass}(see also Lemma 4.1 in \cite{alikakos1994convergence},  Lemma A.1 in \cite{fei2021phase} for the special double well-potential). Here we quote, adapt to our system, and formulate the result as the following statement. 
\begin{lemma} \label{lem-ODE-existence}   Suppose $f=f(z)$ decays exponentially fast to a constant $f^\pm$ as $z\to \pm\infty$, then the system 
    \begin{equation}
        \mathrm L_0 w= f(z), \qquad z\in \mathbb R
    \end{equation}
    has a  solution which decays exponentially fast to $\frac{f^\pm}{W''(\pm 1)}$ as $z\to \pm\infty$ if and only if 
    \begin{equation}\label{cond-compatibility}
        \int_{\mathbb R} f(z)\phi_0'(z)\md z=0. 
    \end{equation}
    Moreover, if the derivatives of $f$ converge to zero, then the derivatives of the solution $w$ also decay exponentially fast to  zero as $|z|\to \infty.$
\end{lemma}
\begin{proof}
For the case when \( f \) converges to zero, the existence of a function \( w \), which is positive and converges to a constant as \( z \to \infty \), follows directly from a similar argument to that in \cite{fei2021phase}, using the properties of \( \phi_{0} \) stated in Lemma~\ref{lem-phi0}.  
For brevity, we omit the details here.

In general, let \( f_{0} \) be any smooth function that connects \( \frac{f^{\pm}}{W''(\pm 1)} \) as \( z \to \pm\infty \) at an exponential rate, and whose derivatives also decay exponentially.  
We then consider
\[
\tilde{w} = w - f_{0},
\]
which satisfies
\begin{equation}
    \mathrm{L}_{0} \tilde{w} = \tilde{f}, 
    \qquad \tilde{f} = f - \mathrm{L}_{0} f_{0}.
\end{equation}
The right-hand side \( \tilde{f} \) remains orthogonal to \( \phi_{0}' \) and converges exponentially to zero as \( z \to \pm\infty \).  
Therefore, the previous argument applies, and the proof is complete.  
\end{proof}}

\subsection{Sovability of the order by order system} 
Through the expansions of $(u_\varep, v_\varep), d_\varep, \sigma_\varep$ and $E_\varep$,  in \eqref{Expansion-u-v}, \eqref{Expansion-d}, \eqref{Expansion-sigma} and \eqref{Expansion-E}, introduced earlier in section 3.1, the order-by-order system $(\mathrm{Eq}_j)_{j=0}^k$ of $(u_j,v_j)$ is given by \eqref{eq-u-v-0},\eqref{eq-uv1}, \eqref{eq-uv2} and \eqref{system-rE_j}. As showed in the appendix, this system has solutions $(u_j,v_j)_{j=1}^k$ defined on $\mathbb R\times\Omega \times [0,T]$ for suitably chosen $(d_j, E_j)$. These functions  decay exponentially fast to constants as $z\to \pm\infty$, and depend on background parameters $(\sigma_j)_{j=0}^k$.

Particularly, for given $(\sigma_j)$ with $\sigma_0=0,\sigma_1=0$, the leading orders solving $(\mathrm{Eq}_j)$  are given by 
\begin{equation}\label{approximate-solution-form-0}
    \begin{aligned}
        &u_0=\phi_0, \quad u_1=0, \quad u_2 =D_0\mathrm L_0^{-1}(z\phi_0'); \\
        & v_0=\Delta d_0\phi_0', \quad v_1=\Delta d_1 \phi_0' -D_0z\phi_0'; 
    \end{aligned}
\end{equation}
The sovability of these equations $(\mathrm{Eq}_j)$ implies the dynamics of $d_j$, and the form of $E_j$ in order to be compatible with the condition $|\nabla d_\varep|=1$. This is given in terms of parameters $(\sigma_l)_{l=1}^{j+1}$.  Particularly, the dynamics of the $d_j$ gives the dynamics of $\Gamma_a$. This is presented below. The leading order dynamics of $\Gamma_0$ is given in the introduction. 
We also introduce 
\begin{equation}
    V_k:=\left( -\Delta^2 d_k +\sum_{0,k}\left(\nabla D_l\cdot \nabla d_{k-l} +D_l\Delta d_{k-l}\right) \right)\Bigg|_{\Gamma_0}.
\end{equation}
 The value of $d_k$ on $\Gamma_0$ satisfies a linear evolution equation as  
\begin{equation}
    \p_t d_k = V_k+\frac{2\sigma_{k+2}}{m_1^2}, 
\end{equation}
for some $\sigma_{k+2}$ given. Then the Eikonal equation $|\nabla d_a|=1$ determines $d_k$ in $\Gamma_0^\ell$.  
Let $d_a=\sum_{j=0}^kd_j$, we define the approximate surface $\Gamma_a(t)$ as 
\begin{equation}\label{def-approx-surface}
    \Gamma_a:=\Big \{x\in \Gamma_0^\ell: , \; d_a(\bm x,t)=0\Big\}. 
\end{equation}
We point out that for each point $\mathbf X_a$ on $\Gamma_a$, there exists a point $\mathbf X_0$ on $\Gamma_0$ such that 
\begin{equation}
    \mathbf X_a(s;t) = \mathbf X_0(s;t)-d_a(\mathbf X_a(s;t),t) \mathbf n_a(\mathbf X_a(s,t)). 
\end{equation}
The reverse is also true. Moreover, 
these background parameters $\{\sigma_{k+2}\}$ are determined by the mass condition \eqref{mass-condition} which gives mass-preserving geometric flows at different orders.  This shall be discussed in the following section. 
\subsection{Gluing solution and mass condition}
The gluing method is used to construct approximate solutions by smoothly combining the behavior of the solution near the interface (described by the local coordinate  $z$) and far from the interface (described by the far-field behavior). 
\begin{defn}
    For a given function $u=u(z,x,t)$ defined on $\mathbb R\times \Omega\times (0,T)$ which decays exponentially to $u^\pm(t)$ as $z\to \pm \infty$,  we define its glued form, $u^g=u^g(x,t)$, as: for $z=z(x,t)$ 
\begin{equation}\label{def-gluing}
    u^g(x,t)= u^+(t) \chi^+\left(\frac{\varep z}{\ell}\right) + u^-(t) \chi^-\left(\frac{\varep z}{\ell}\right) +u(z,x,t)\left(1-\chi^+\left(\frac{\varep z}{\ell}\right) -\chi^-\left(\frac{\varep z}{\ell}\right)\right).
\end{equation}
Here we have introduced the smooth cut-off function satisfying
\begin{equation}
    \chi^+(r)= \left\{\begin{aligned}
       & 1, \qquad r>1; \\
        & 0, \qquad r<1/2. 
    \end{aligned}\right. 
\end{equation}
and $\chi^-(r)=\chi^+(-r)$.
\end{defn}
The background parameters $\{\sigma_k\}_k$ are determined by the mass condition. Specifically, the mass condition dictates the surface area, which in turn governs the dynamics of the background state. This process is detailed below for the leading-order approximation, with higher-order corrections handled inductively as outlined in the appendix.   
\begin{lemma}
    Suppose $u_0^g$ satisfies the mass condition up to order $\varep$, then the leading order volume of the enclosed region, denoted as $\mathcal V_0$, is 
    \begin{equation}
        \mathcal  V_0 =\frac{1}{2}(|\Omega|-M_0). 
    \end{equation}
\end{lemma}
\begin{proof}
    Let $\Gamma_0$ be the leading order surface, the enclosed region is $\Omega^-$ and the exterior region is $\Omega^+.$ In the enclosed region, we have $\chi^+=0$ and  
    \begin{equation}
        \int_{\Omega^-} u_0^g \, dx= \int_{\Omega^-} \Big((u_0-u_0^-)(1-\chi^-) +u_0^- \Big ) dx. 
    \end{equation}
    Note that $u_0$ converges to  $u_0^-=-1$ exponentially as $z\to -\infty$, using $|\Omega^-|=\mathcal V_0$ implies 
    \begin{equation}
        \int_{\Omega^-} u_0^g \, dx  = -\mathcal  V_0+O(\varep). 
    \end{equation}
    Similarly using $u^+=1$ and $|\Omega^+|=|\Omega|-\mathcal V_0$ yields 
    \begin{equation}
        \int_{\Omega^+} u_0^g \, dx  = |\Omega|-\mathcal V_0+O(\varep). 
    \end{equation}
    The Lemma follows from the domain decomposition $\Omega=\Omega^+\cup \Omega^-$ and the mass condition: 
    \begin{equation}
        \left|\int_\Omega u_0^g \,dx-M_0\right| \leq C\varep. 
    \end{equation}
    This completes the proof. 
\end{proof}
The background state $\sigma_2$ is determined if the volume  of the leading order surface, $\Gamma_0$ determined  by \eqref{def-dynamics-Gamma0}, is fixed.
\begin{lemma}
    The geometric flow of $\Gamma_0=\Gamma_0(t)$, $\{x\in \Omega: d_0=0\}$, determined by $\mathrm G_0[d_0,\sigma_2]=0$  has fixed enclosed volume if $\sigma_2$ is given by 
    \begin{equation}
        \sigma_2=  \frac{m_1^2}{|\Gamma_0|} \int_{\Gamma_0} \mathrm V_0\, ds.  
    \end{equation}
\end{lemma}
\begin{proof}
    Note that  $\Gamma_0$  has fixed volume. This volume conservation implies 
\begin{equation}
    0=\frac{d}{dt} \mathrm{Vol}(\Gamma_0)= \int_{\Gamma_0}  \mathbf n_0\cdot \p_t \mathbf X_0 \, ds. 
\end{equation}
Combining the identity with the dynamics of $\Gamma_0$ given by \eqref{def-dynamics-Gamma0} implies
\begin{equation}
    \frac{\sigma_2}{m_1^2} |\Gamma_0| = \int_{\Gamma_0} \mathrm V_0\, ds. 
\end{equation}
Solving for $\sigma_2$, we obtain the result in the Lemma. This completes the proof. 
\end{proof}


\subsection{Existence of approximate solutions}For clarity and convenience, we do not distinguish between two similar sets of functions,  $u_j$  and  $u_j^g$. The difference between these two sets of functions is stated to be exponentially small, meaning that, while  $u_j $ and  $u^g_j$  might differ slightly, this difference decays extremely quickly as the parameter, $|z|$, increases, making the distinction negligible for the purposes of the proof. 
\begin{thm}[Existence of $k$-approximate solutions] Let $(\Gamma_0, T)$ be a compatible data as in Definition \ref{def-compatible-data}.   Let $k\geq 1$ be any positive integer, then there exists a $k$-approximate solution $(u_a, v_a)$ to the system \eqref{eq-2nd-order-system}-\eqref{mass-condition}  as introduced in Definition \ref{def-k-approximate sol}.  
\end{thm}
\begin{proof}
    By solving the system  $\mathrm (\mathrm{Eq}_j)$, given in the appendix, one can define the approximate solutions  $u_a=\sum_{j=0}^ku_j$, $v_a=\sum_{j=0}^kv_j$ and $\sigma_a = \sum_{j=0}^{k-1}\sigma_j$. These are functions defined on  $\mathbb{R} \times \Omega \times [0, T]$, and decay exponentially fast to constants as $|z|\to\infty$. 
    
    Now we introduce the $k$-approximate solution via the gluing method, $u_a^g$ as in Definition \ref{def-gluing}, which is defined on $\Omega\times [0,T]$. This method ensures that the solution is smoothly joined across the interface and away from it, with the cut-off functions $ \chi^\pm(\varepsilon z/\ell)$  ensuring a smooth transition between the near-interface and far-field regions.
    
    The exponentially fast decay of the solutions  $u_j$  and  $v_j$  as  $|z| \to \infty$  ensures that the errors introduced by the gluing process are exponentially small. These small errors can be absorbed into the residuals of the system, allowing the gluing solution to qualify as a k-approximate solution, see Definition  \ref{def-k-approximate sol}.    
\end{proof}

\section{Proof of Convergence}\label{sec-main-theorem}
We consider  the mass-preserving $L^2$-gradient flow \eqref{eq-FCH-L2}, the solution to the flow is denoted as $u_\varep$. Let $u_a$ be an $k$-approximate solution introduced in Section \ref{sec-approximate-solutions}. In this section we establish their difference estimate in terms of orders of $\varep$. This estimate provides a limiting estimate when taking $\varep$ to zero, which gives a rigorous justification of the limit geometric flow. 


 Introducing the error $u=u_\varep-u_a$,  the flow for $u$ can be written as: 
\begin{equation}\label{flow-decomposition}
    \p_t u +\Pi_0 \mathbb L_\varep u =\mathcal R(u_a) +\Pi_0 \mathcal N(u).
\end{equation}
Here, $\mathbb L_\varep = \frac{\delta^2 \mathcal F}{\delta u^2}|_{u=u_a}$ is the linearized operator.  If we introduce the notion of the  linear operator 
\begin{equation}\label{def-rL} 
    \mathrm L_\varep:=  \Delta -\varep^{-2}W''(u_a),
\end{equation}
the full linearized operator,  $\mathbb L_\varep$, takes the form 
\begin{equation}\label{form-bL}
     \varep^4 \mathbb L_\varep= \varep^4 \mrL_\varep^2  +\varep \mathrm R_a, 
\end{equation}
where the first  term, $\varep^4 \mrL_\varep^2$, represents the highest-order contribution to the linearized operator;  The term  $\mathrm R_a$  is a remainder term involving the functions  $u_a$  and  $v_a$, defined as:
\begin{equation}\label{def-Ra}
    \mathrm R_a:= -v_aW'''(u_a) . 
\end{equation} 
Moreover, the term  $\mathcal R$, in \eqref{flow-decomposition}, is the  residual  capturing the accuracy of the approximate solution $u_a$. It is defined by:
\begin{equation}\label{def-residual}
    \mathcal R(u_a) := - \p_t u_a-\frac{1}{\varep} \Pi_0 \mrF(u_a); 
\end{equation}
The nonlinear term $\mathcal N$ captures the nonlinear interactions of the error term  $u$  with the flow. It is given by:
\begin{equation}\label{def-nonlinear-term}
    \mathcal N(u) := \frac{1}{\varep} \left(-\mrF(u+u_a)+\mrF(u_a)\right) + \mathbb L_\varep u. 
\end{equation}
To relate the fourth-order equation residual  $\mathcal{R}=\mathcal R(u_a)$  to the second-order system residuals  $\mathcal{R}_1$  and  $\mathcal{R}_2$, introduced in \eqref{def-remainder-R12}, we can express this relation as follows:
\begin{equation}\label{relation-remainders}
    \mathcal R = \frac{\mathcal R_1}{\varep^3} - \Pi_0\left[\mathrm L_\varep\frac{\mathcal R_2}{\varep^2}\right]. 
\end{equation}

\subsection{Coercivity of the linearized operator}
The coercivity of the linearized operator, $\mathbb L_\varep$, depends on analyzing its near-zero spectrum. The near-zero modes describe lateral displacements or movements of the sharp interface, leading to instabilities that cause the interface to “meander.”  The coercivity of  $\mathbb L_\varepsilon $ requires controlling these near-zero modes, which are challenging because of their proximity to zero in the spectrum. This involves a careful spectral analysis of the linearized operator to ensure that even though these modes are near zero, the operator still maintains enough positivity to prevent instability in finite time. 

The function  $\varphi(z, s)$  is introduced as part of the analysis of the linearized operator.  It is defined in terms of a cut-off function  $\zeta$, which helps localize the function  $\varphi$  near the interface  $\Gamma$. The function  $\varphi$  has the form:
\begin{equation}\label{def-varphi}
\varphi(z,s):= \varep^{-\frac{1}{2}}\phi_0'\zeta(\varep z/\ell),\qquad \zeta=1-\chi^+-\chi^-.
\end{equation}
Here we recall that $\chi^\pm$ are cut-off functions defined in Definition \ref{def-gluing}. 

The function  $\varphi(z, s)$  is  localized near the interface  $\Gamma$.  Localization means that the function decays rapidly as you move away from the interface, and this is formally described by the condition: for any positive integer $k$, there exists a positive constant $C=C(k)$ depending on $k$ such that 
\begin{equation*}
|\p_z^jf(  x(s, z))|\leq C e^{-\nu|z|}, \qquad j\leq k
\end{equation*}
for some constants  $\nu > 0$. Here, $\nu$ is referred as the decaying exponent of $f$. This exponential decay ensures that the function is concentrated near the interface and effectively vanishes far from it. 

Given a function  $u \in H^2(\Omega)$, it is useful to decompose  $u$  into components along different directions: the  parallel part, denoted as $u^\parallel$, and a remaining orthogonal part  ${w}$. This decomposition is expressed as:
\begin{equation}
   u= u^\parallel + {w}, \quad \hbox{where} \quad u^\parallel:= Z(s) \varphi\left(z,s\right) 
\end{equation} 
with $Z=Z(s)$  determined by the orthogonal condition below: 
\begin{equation}\label{orthogonal-decomposition-Z}
    \int_{-\ell/\varep}^{\ell/\varep} {w} \varphi\left(z,s\right) J(\varep z,s) dz=0.  
\end{equation} 
The orthogonality implies the following result. 
\begin{lemma}\label{lem-Z-u-L2}
      Let $u= u^\parallel + {w}$ be decomposed as \eqref{orthogonal-decomposition-Z} for some $(Z,{w})$.  Then  there exist some positive constants $C_1, C_2$ independent of $\varep$ such that
    \begin{equation}
        C_1\left(\|Z\|_{L^2(\Gamma)} +\|{w}\|_{L^2}\right)\leq\|u\|_{L^2}\leq C_2\left(\|Z\|_{L^2(\Gamma)} +\|{w}\|_{L^2}\right). 
    \end{equation}
\end{lemma}
\begin{proof}
    From the decomposition $u= u^\parallel + {w}$(\eqref{orthogonal-decomposition-Z}), we begin by expressing the  $L^2 $-norm of $ u $ as:
    \begin{equation}
        \int_\Omega u^2 dx = \int_\Omega {w}^2 dx +2\int_\Omega {w} u^\parallel dx + \int_\Omega |u^\parallel|^2 dx.  
    \end{equation} 
   By the orthogonality condition (from equation \eqref{orthogonal-decomposition-Z}), the cross term  vanishes, so we have:
\begin{equation}
\int_\Omega u^2 \, dx = \int_\Omega {w}^2 \, dx + \int_\Omega |u^\parallel|^2 \, dx
\end{equation}
This implies that  $\| {w} \|_{L^2} \leq \| u \|_{L^2}$. The next step is to estimate  $\displaystyle \int_\Omega |u^\parallel|^2 \, dx$. Since  $u^\parallel$  is supported in the tubular neighborhood  $\Gamma^\ell $ (around the interface), we express this term in local coordinates:
    \begin{equation}
        \int_\Omega |u^\parallel|^2 dx = \int_{-\ell/\varep}^{\ell/\varep}\int_\Gamma |\phi_0'|^2Z^2\zeta^2J(\varep z,s)dsdz.
    \end{equation} 
   The localized function  $\phi_0^{\prime}$    is exponentially small away from the interface  $\Gamma$, using the definition of $m_1$ and the estimate of the Jacobian in \eqref{est-Jacobian} implies for $\varep$ small enough the following inequality holds: 
   \begin{equation}
       \frac12 m_1^2\|Z\|_{L^2(\Gamma)}^2\leq \int_\Omega |u^\parallel|^2 dx \leq 2m_1^2\|Z\|_{L^2(\Gamma)}^2.  
   \end{equation}
   This completes the proof. 
\end{proof}

To measure the size of functions in a space adapted to the problem’s scaling, the inner norm  $H^2_{\text{in}}$  is introduced. This norm weights the derivatives of  $u$  according to their order, with higher derivatives scaled by increasing powers of  $\varepsilon$. It  is defined as:
\begin{equation}\label{def-inner-norm}
    \|u\|_{\Htwoin}^2  : = \sum_{k=0}^2\|\varep^k \nabla^k u\|_{L^2}^2.  
\end{equation}
For $u^\parallel$ in \eqref{orthogonal-decomposition-Z}, using the expansion of the Laplace operator in \eqref{expansion-Laplace} and interpolation inequality implies 
\begin{equation}\label{est-u-parallel-Htwoin}
    \|u^\parallel\|_{\Htwoin}^2\leq C\varep^4 \|Z\|_{H^2(\Gamma)}^2+\|Z\|_{L^2(\Gamma)}^2. 
\end{equation}
From the definition of the linear operator $\mathrm L_\varep$,  there exist positive constants $C_1, C_2, C_3$ such that  for any $u\in H^2$, the following holds:  
    \begin{equation} \label{est-rLu-Htwoin}
       C_1 \|u\|_{\Htwoin}-C_2 \|u\|_{L^2}\leq  \|\varep^2 \mathrm L_\varep u\|_{L^2}\leq C_3\|u\|_{\Htwoin}.
    \end{equation}
  Below we state a coercivity result for the linearized operator  $\mathbb L_\varepsilon$  in the context of a $k$-approximate solution. 
\begin{thm}[Coercivity]\label{thm-coercivity} Let $k\geq 1$ be a given integer, $\mathbb L_\varep$ be the linearized operator,  \eqref{form-bL}, around a $k$-approximate solution $u_a$. For any $u\in H^2$ it can be decomposed as $u=u^\parallel+{w}$ by \eqref{orthogonal-decomposition-Z} for some $(Z, {w})$. Moreover,   the following coercivity holds for the linearized operator $\mathbb L_\varep$, 
    \begin{equation*}
         \left<\mathbb L_\varep u, u\right>_{L^2}\geq C_1\left(\varep^{-4} \|{w}\|_{\Htwoin}^2+ \|Z\|_{H^2(\Gamma)}^2\right) -C_2\|u\|_{L^2}^2 
    \end{equation*} 
   for some positive constants $C_1, C_2$ independent of $\varep$. Furthermore, for updated $C_k$s 
    \begin{equation*}
    \begin{aligned} 
         \left<\Pi_0\mathbb L_\varep u, u\right>_{L^2}& \geq \frac12\left<\mathbb L_\varep u, u\right>_{L^2}+C_1\left(\varep^{-4} \|{w}\|_{\Htwoin}^2+\|Z\|_{H^2(\Gamma)}^2\right) \\
         &\quad -C_2\|u\|_{L^2(\Gamma)}^2
         -\frac{C_2}{\varep^6} \left(\int_\Omega udx\right)^2.  
         \end{aligned} 
    \end{equation*}
\end{thm}
The coercivity result guarantees that the quadratic form  $\langle \mathbb L_\varepsilon u, u \rangle$  is bounded from below by a positive definite form involving the norms of  ${w}$  and  $Z$, minus some lower-order terms. This means that the linearized operator is coercive, which implies that the operator is stable and the perturbations do not grow uncontrollably. For clarity of the presentation, the proof of this coercivity Theorem is postponed  and outlined in  Section \ref{sec-linear-coercivity}. 

\subsection{Estimates of nonlinearity} 
We can rewrite the nonlinear term $\mathcal N$,  \eqref{def-nonlinear-term}, using the definition of  the chemical potential in \eqref{eq-FCH-L2-p} as: 
\begin{equation}\label{form-explicit-nonlinear-term}
\begin{aligned} 
    \varep^4 \mathcal N(u) =&  \left(W''(u_\varep)-W''\right)\varep^2\mathrm L_\varep u +\varep^2\mathrm L_\varep  (W'(u_\varep) -W'-W''u) \\
    & + (W''(u_\varep)-W''-W'''u) (\varep^2\Delta u_a-W'(u_a) ). 
    \end{aligned}
\end{equation}
Here functions $W,W', W''$ without specifying take values at the $k$-approximate solution $u_a.$ 
Before establishing the bound of the nonlinear term, we provide a useful lemma. 
\begin{lemma}\label{lem-nonlinear-est-aux}
    Let $g=g(u)$ be a polynomial function of $u$ with degree $2\beta-2$ and  satisfying $g(0)=0, g'(0)=0$, then for any $\delta, r_0>0$ there exists a positive constant $C_\delta$, depending on $\delta$ only, such that 
\begin{equation}
   \left| \int_\Omega g(u)  \Delta u dx\right| \leq  \delta \int_\Omega |\nabla u|^2u^{2\beta-2} dx +C_\delta \int_\Omega  |\nabla u|^2 |u|^{r_0} dx.
\end{equation}
\end{lemma}
\begin{proof}
   Integration by parts implies   
 \begin{equation}
     \int_\Omega  g(u)\Delta u dx = -\int_\Omega g'(u)|\nabla u|^2 dx. 
 \end{equation}
 Since $g$ is a polynomial of $u$ with degree $2\beta-2$, then $g'$ has degree $2\beta-3$. Moreover, $g'(0)=0$ implies  
 \begin{equation}\label{pre-est-g-0} 
     \left| \int_\Omega  g(u)\Delta u dx\right| \leq C\int_\Omega |\nabla u|^2(|u|+|u|^{2\beta-3}) dx . 
 \end{equation}
Let $\delta,r_0\in(0,1)$  be any given constants, then  we claim that point-wisely it holds that 
\begin{equation}\label{pre-est-g-1} 
    |u| +|u|^{2\beta-3}\leq \delta u^{2\beta-2} +2\left(\frac{1}{\delta}\right)^{2\beta-3-r_0}|u|^{r_0}. 
\end{equation}
This is true by considering $\{|u|\geq 1/\delta\}$ and $\{|u|<1/\delta\}$. In fact, taking into account the case $\{|u|\geq 1/\delta\}$ with $\delta\in(0,1)$, for $2\beta\geq 4$ we have
\begin{equation}
    |u|\leq |u|^{2\beta-3}\leq  \delta|u|^{2\beta-2} . 
\end{equation}
Similarly, one can show that for the case $\{|u|<\frac{1}{\delta}\}$, it holds that 
\begin{equation}
  |u|+|u|^{2\beta-3}= |u|^{r_0} \left( |u|^{1-r_0}+|u|^{2\beta-3-r_0}\right)\leq 2\left(\frac{1}{\delta}\right)^{2\beta-3-r_0}|u|^{r_0}.  
\end{equation}
The claim \eqref{pre-est-g-1} holds. Therefore, returning the estimate \eqref{pre-est-g-1} to \eqref{pre-est-g-0} implies
    \begin{equation}
     \left| \int_\Omega g(u)\Delta u dx\right| \leq \delta \int_\Omega |\nabla u|^2u^{2\beta-2} dx +\frac{2}{\delta^{2\beta-3-r_0}} \int_\Omega  |\nabla u|^2 |u|^{r_0} dx . 
 \end{equation}
The Lemma follows.
\end{proof}
We first establish a bound for the mass of the nonlinear term $\mathcal N.$ 
\begin{lemma}\label{lem-mass-nonlinear-term}
    Suppose that  $W=W(\phi)$ is a polynomial of degree $2\beta\geq 4$ and takes the form \eqref{assumption-W}, then for any $\delta, r_0\in(0,1)$, the nonlinear term $\mathcal N(u)$, see \eqref{form-explicit-nonlinear-term}, there exists a positive constant $C_\delta$ which might depend on $\delta$ such that 
    \begin{equation}
        \varep^4 \left|\int_\Omega \mathcal N(u)dx\right|\leq 2\delta \varep^2 \int_\Omega |\nabla u|^2 u^{2\beta-2} dx+C_\delta \varep^2 \int_\Omega |\nabla u|^2 |u|^{r_0}dx +C\int_\Omega (|u|^2 +|u|^{2\beta-1})dx. 
        \nonumber
     \end{equation}
\end{lemma}
\begin{proof}
 From the equation \eqref{form-explicit-nonlinear-term}, using integration by parts we derive 
 \begin{equation}\label{est-integral-cN-0}
 \begin{aligned}
     \varep^4 \int_\Omega \mathcal N(u)dx =& \int_\Omega  \left(W''(u_\varep)-W''\right)\varep^2\mathrm L_\varep u dx -\int_\Omega   W''  (W'(u_\varep) -W'-W''u)dx \\
    & + \int_\Omega(W''(u_\varep)-W''-W'''u) (\varep^2\Delta u_a-W'(u_a) )dx
    \end{aligned}
 \end{equation}
From the definition of the linear operator $\mathrm L_\varep$, the first term on the right hand side can be rewritten as 
 \begin{equation}
     \int_\Omega  \left(W''(u_\varep)-W''\right)\varep^2\mathrm L_\varep u dx = \int_\Omega (W''(u_\varep)-W) \varep^2 \Delta u dx - \int_\Omega W''(W''(u_\varep)-W'')u dx. 
 \end{equation}
 Plugging into back to the previous identity \eqref{est-integral-cN-0} implies 
 \begin{equation}
     \begin{aligned}
     \varep^4 \int_\Omega \mathcal N(u)dx =& \int_\Omega  \left(W''(u_\varep)-W''\right)\varep^2\Delta u dx -\int_\Omega   W''  (W'(u_\varep) -W'-W''u+(W''(u_\varep)-W'')u)dx \\
    & + \int_\Omega(W''(u_\varep)-W''-W'''u) (\varep^2\Delta u_a-W'(u_a) )dx
    \end{aligned} 
 \end{equation}
Since $W$ is a polynomial of degree $2\beta$ and $u_a$ is a smooth uniformly bounded function, we have
    \begin{equation}\label{est-W}
    \begin{aligned}
        &|W''(u_\varep)-W''(u_a)|\leq C(|u|+|u|^{2\beta-2}); \\ &|W'(u_\varep) -W'-W''u|\leq C(u^2+|u|^{2\beta-1});\\ 
        &|W''(u_\varep)-W''-W'''u|\leq C(u^2+|u|^{2\beta-2}).
        \end{aligned}
    \end{equation}
    Consequently, we have 
    \begin{equation*}
 \begin{aligned}
     \varep^4 \left| \int_\Omega \mathcal N(u)dx\right| \leq \left|\int_\Omega  \left(W''(u_\varep)-W''\right)\varep^2\Delta u dx\right|+ C\int_\Omega  (|u|^{2\beta-1}+|u|^2) dx.  
    \end{aligned}
 \end{equation*} 
Applying Lemma \ref{lem-nonlinear-est-aux} implies 
    \begin{equation}
     \left| \int_\Omega  \left(W''(u_\varep)-W''\right)\varep^2\Delta u dx\right| \leq 2\delta \varep^2\int_\Omega |\nabla u|^2u^{2\beta-2} dx +C_\delta \int_\Omega \varep^2 |\nabla u|^2 |u|^{r_0} dx . 
 \end{equation}
The Lemma follows from the previous two estimates.  
\end{proof}

\begin{lemma}\label{lem-aux-GN} 
    Let $p\geq 2$ be any given positive constant, then there exists a positive constant $C>0$ such that 
    \begin{equation}
        \|u^p\|_{L^2}^2 \, \leq\, C \|\nabla u^p\|_{L^2}^{2\theta} \left(\|u\|_{L^2}^2 +\|u\|_{L^2}\|u^p\|_{L^2} \right)^{2(1-\theta)}+ C\left(\|u\|_{L^2}^4 +\|u\|_{L^2}^2\|u^p\|_{L^2}^2\right). 
    \end{equation}
    Particularly, assume further  $\|u\|_{L^2}^2\leq C\varep^5$, then  for any $\delta>0$, there exists $C_\delta$ and $\varep_0>0$(depending on $\delta$) such that 
     \begin{equation}
         \varep^{-2}\|u^p\|_{L^2}^2 \leq \delta \varep^2 \|\nabla u^p\|_{L^2}^2 +C_\delta\varep^{-6} \|u\|_{L^2}^4.  
    \end{equation}
\end{lemma}
\begin{proof}
    Applying the Gargliardo-Nirenberg inequality yields  
\begin{equation}\label{est-GN-0} 
\|u^p\|_{L^2}^2\leq C\|\nabla u^p\|_{L^2}^{2\theta}\|u^p\|_{L^1}^{2(1-\theta)} +C\|u^p\|_{L^1}^2, \qquad \theta= \frac{N}{N+2}. 
\end{equation}
Observe that $2\theta=\frac{2N}{N+2}<2$. Applying H\"older's inequality implies
\begin{equation}
    \int_\Omega |u|^p\, dx = \int_\Omega |u|^{p-1}|u|\,dx \leq  \left(\int_\Omega u^{2(p-1)} dx\right)^{1/2}\left( \int_\Omega u^2 dx\right)^{1/2}.
\end{equation} 
Using the inequality   $|u|^{2(p-1)}\leq C|u|^2+C|u|^{2p}$ for $2(p-1)\in (2,2p)$ implies
\begin{equation}\label{est-GN-1}
    \int_\Omega |u|^p\, dx \leq \|u\|_{L^2}\left(\|u\|_{L^2}+\|u^p\|_{L^2}\right). 
\end{equation}  
The first inequality in the lemma follows by returning the inequality above \eqref{est-GN-1} to \eqref{est-GN-0}. The second inequality is a direct application of the first inequality and Young's inequality. More precisely, applying Young's inequality implies 
\begin{equation}
    \varep^{-2}\|\nabla u^p\|_{L^2}^{2\theta} \left(\|u\|_{L^2}^2 +\|u\|_{L^2}\|u^p\|_{L^2} \right)^{2(1-\theta)} \,\leq\,  \delta \varep^2 \|\nabla u^p\|_{L^2}^2 +\frac{C}{\delta \varep^6} \left(\|u\|_{L^2}^2 +\|u\|_{L^2}\|u^p\|_{L^2} \right)^2. 
\end{equation}
Applying the Cauchy-Schwarz inequality to the second term on the right-hand side of the above inequality,  combining with the first inequality,  implies
\begin{equation}
    \varep^{-2}\|u^p\|_{L^2}^2  \,\leq\,  \delta\varep^2  \|\nabla u^p\|_{L^2}^2 +\frac{C}{\delta \varep^6} \left(\|u\|_{L^2}^4 +\|u\|_{L^2}^2\|u^p\|_{L^2}^2 \right). 
\end{equation} 
 The second inequality follows from the inequality above under the assumption of $\|u\|_{L^2}^2$ provided with $\varep_0$ small enough(depending on $\delta$). This completes the proof.  
\end{proof}
Next we bound the inner product of the nonlinear term $\mathcal N(u)$ with $u$. 

\begin{lemma}\label{lem-nonlinear-est} Let $k\geq 1$ and $u_a$ be a $k$-approximate solution.  Suppose that  $W=W(\phi)$ is a polynomial of degree $2\beta\geq 4$ and takes the form \eqref{assumption-W}, then for any $\delta>0$, the nonlinear term $\mathcal N(u)$, see \eqref{form-explicit-nonlinear-term}, admits the following bound: 
\begin{equation}
   \left< \Pi_0\mathcal N(u), u \right>_{L^2}\leq \delta\left<\mathbb L_\varep u, u\right>_{L^2}+C \|u\|_{L^2}^2+ \frac{C_\delta}{\varep^{10}} \|u\|_{L^2}^{4} +\frac{C_\delta}{\varep^4} \|u\|_{L^2}^{2+r_0}. 
\end{equation}
Here $r_0=\min\{1, \frac{4}{n}\}$. 
\end{lemma}
\begin{proof}
Using  integration by parts and  after some algebraic rearrangements, we rewrite  
\begin{equation}
    \varep^4 \left< \Pi_0\mathcal N(u), u \right>_{L^2} = \mathscr N_0+\mathscr N_1+ \mathscr N_2, 
\end{equation}
where the terms $\mathscr N_0, \mathscr N_1, \mathscr N_2$ shown  on the right-hand side are defined by: 
\begin{equation}
    \begin{aligned}
    \mathscr N_0 &:= - \frac{\varep^4}{|\Omega|}\int_\Omega \mathcal N(u) dx \int_\Omega u dx; \\ 
       \mathscr N_1 & := \int_\Omega \varep^2 \mrL_\varep u \left((W''(u_\varep) -W'')u+W'(u^\varep) -W'-W''u\right)  dx; \\
      \mathscr  N_2& := \int_\Omega \left(\varep^2 \Delta u_a-W'(u_a)\right) \left(W''(u_\varep)-W''-W'''u\right)udx.
    \end{aligned}
\end{equation}
 Note that for a $k$-approximate solution $u_a$, the mass of $u=u_\varep-u_a$ is in the order of $\varep^{k+1}.$ With the aid of Lemma \ref{lem-mass-nonlinear-term}, the first term $\mathscr N_0$ can be bounded by 
\begin{equation}\label{est-sN0}
    \mathscr N_0  \leq C\varep^{k+1}\left(\varep^2 \int_\Omega |\nabla u|^2 u^{2\beta-2} \,dx + \varep^2 \int_\Omega |\nabla u|^2 |u|^{r_0} \, dx+\int_\Omega (|u|^2+|u|^{2\beta-1})\, dx\right). 
\end{equation}
Since $W$ has the form in \eqref{assumption-W}, there exists some polynomial $c_1=c_1(u)$ with degree ${(2\beta-2)}$ such that 
\begin{equation}
    (W''(u_\varep) - W'')u+W'(u_\varep) -W'-W''u=4c_0\beta^2 u^{2\beta-1} +c_1(u). 
\end{equation}
Consequently, the first term $\mathscr N_1$ reduces to the following form: 
\begin{equation}
   \mathscr N_1 = 4c_0\beta^2 \varep^2 \int_\Omega  u^{2\beta-1}\mrL_\varep u dx +\int_\Omega c_1(u) \varep^2 \mrL_\varep u dx.
\end{equation}
Using the definition of $\mrL_\varep$, \eqref{def-rL}, and integration by parts implies: 
\begin{equation}
    \begin{aligned}
     \mathscr   N_1&= -4\beta^2(2\beta-1)c_0 \int_\Omega \varep^2 |\nabla u|^2 u^{2\beta-2} dx- 4\beta^2c_0\int_\Omega W''(u_a)u^{2\beta} dx+\int_\Omega c_1(u) \varep^2 \mathrm L_\varep u dx. 
     \nonumber
    \end{aligned}
\end{equation} 
Let $r_0>0$ be any given positive constant.  Applying Lemma \ref{lem-nonlinear-est-aux} implies for any $\delta>0$ there exists a positive constant $C_\delta$ such that 
\begin{equation}
   \left| \int_\Omega c_1(u) \varep^2 \Delta u dx\right| \leq  2\delta \varep^2\int_\Omega |\nabla u|^2u^{2\beta-2} dx +C_\delta \int_\Omega \varep^2 |\nabla u|^2 |u|^{r_0} dx.
\end{equation}
Note that $u_a$ is smooth and uniformly bounded, the potential $W$ taking value at $u_a$ is also bounded. Using this fact and choosing $\delta$ small enough, depending on $\beta$ and $c_0$ only, implies
\begin{equation}\label{est-N1}
    \begin{aligned}
    \mathscr  N_1&\leq  -2\beta^2(2\beta-1)c_0 \int_\Omega \varep^2 |\nabla u|^2 u^{2\beta-2} dx +C_\delta \int_\Omega \varep^2 |\nabla u|^2 |u|^{r_0} dx+C\int_\Omega u^{2\beta}dx. 
    \end{aligned}
\end{equation} 
Now we handle the $\mathscr N_2$-term by rescaling the estimate \eqref{est-W}. 
For a $k$-approximate solution $u_a$, with $u_0=\phi_0,$ we also have $|\varep^2\Delta u_a-W'(u_a)|\leq C \varep$. Combining these facts above yields
\begin{equation}\label{est-N2}
    \mathscr N_2\leq C\varep \int_\Omega \left(u^3+|u|^{2\beta-1}\right) dx\leq C\varep^4 \|u\|_{L^2}^2 +\frac{C}{ \varep^2} \int_\Omega (u^4+u^{2\beta})dx. 
\end{equation}
Combining the estimates for $\mathscr N_0$, in \eqref{est-sN0}, for $\mathscr N_1$, in  \eqref{est-N1}, and for the term $\mathscr N_2$, in \eqref{est-N2}, yields 
\begin{equation}\label{est-nonlinear-term-1} 
\begin{aligned}
    \varep^4 \int_\Omega \Pi_0\mathcal N(u)udx \leq & -2\beta^2(2\beta-1)c_0 \int_\Omega \varep^2 |\nabla u|^2 u^{2\beta-2} dx +C_\delta \int_\Omega \varep^2 |\nabla u|^2 |u|^{r_0} dx\\
    &+\frac{C}{\varep^2}\int_\Omega \left(u^4+u^{2\beta}\right)dx.
    \end{aligned} 
\end{equation}
Using H\"older's inequality implies for any $p\in(2,\infty)$ it holds that 
\begin{equation}
    \int_\Omega \varep^2 |\nabla u|^2 |u|^{r_0} dx \leq \varep^2 \left(\int_\Omega |\nabla u|^{p}\right)^{\frac{2}{p}} \left(\int_\Omega |u|^{r_0\cdot \frac{p}{p-2}}\right)^{\frac{p-2}{p}} \leq  \varep^2 \|\nabla u\|_{L^p}^2 \|u\|_{L^{\frac{r_0p}{p-2}}}^{r_0}. 
\end{equation}
Choose $p\leq \frac{2n}{n-2}$, applying the Gargliardo-Nirenberg inequality implies
\begin{equation}
    \varep^2 \|\nabla u\|_{L^p}^2 \leq C \varep^2 \|\Delta u\|_{L^2}^{2\theta}\|u\|_{L^2}^{2(1-\theta)}+C\varep^2\|u\|_{L^2}^2, \qquad 2\theta = 1+n\left(\frac12 -\frac1p\right)\leq 2. 
\end{equation}
Choosing $r_0=r_0(p)=\frac{2(p-2)}{p}\leq \frac{4}{n}$ implies
\begin{equation}
     \int_\Omega \varep^2 |\nabla u|^2 |u|^{r_0} dx \leq C\varep^2 \|\Delta u\|_{L^2}^{2\theta}\|u\|_{L^2}^{2(1-\theta)+r_0}+C\varep^2\|u\|_{L^2}^{2+r_0}.
\end{equation}
Applying Young's inequality yields for any $\delta>0$ there exists  a positive constant $C_\delta$ such that 
\begin{equation}\label{est-cN-revised-1}
    \int_\Omega \varep^2 |\nabla u|^2 |u|^{r_0} dx \leq \delta\varep^4 \|\Delta u\|_{L^2}^2 +C_\delta\|u\|_{L^2}^{2+r_0}. 
\end{equation}
In addition, under the assumption $\|u\|_{L^2}\leq \frac{\varep^2}{2C_\delta}$ for some constant $C_\delta$ large enough depending on parameter $\delta$ only Lemma \ref{lem-aux-GN} implies 
\begin{equation}\label{est-cN-revised-2}
     \|u^\beta\|_{L^2}^2\leq \delta \varep^2 \|\nabla u^\beta\|_{L^2}^2 +\frac{C_\delta}{\varep^6} \|u\|_{L^2}^4. 
\end{equation}
Similarly, applying Lemma \ref{lem-aux-GN} to $u^2$ we have
\begin{equation}
    \|u^2\|_{L^2}^2\leq \delta \varep^2 \|\nabla u^2\|_{L^2}^2 +\frac{C_\delta}{\varep^6} \|u\|_{L^2}^4. 
\end{equation}
Note that $|u|\leq |u|^{r_0} +|u|^{2\beta-2}$ for $r_0\leq 1$. Therefore 
\begin{equation}
    \|u^2\|_{L^2}^2\leq \delta \varep^2 \int_\Omega |u|^{r_0}|\nabla u|^2\, dx +\frac{C_\delta}{\varep^6} \|u\|_{L^2}^4. 
\end{equation} 
Combining this inequality \eqref{est-cN-revised-2} and inequality \eqref{est-cN-revised-1} with \eqref{est-nonlinear-term-1}, and choosing $\delta$ small enough yields 
\begin{equation}\label{est-nonlinear-term-2}
    \varep^4 \int_\Omega \Pi_0 \mathcal N(u)udx \leq \delta \varep^4 \|\Delta u\|_{L^2}^2+ \frac{C_\delta}{\varep^6} \|u\|_{L^2}^{4} +C_\delta \|u\|_{L^2}^{2+r_0}.
\end{equation}
The  Lemma follows from the inequality above with an updated constant $C$. The proof is complete. 

\end{proof}

\subsection{Proof of  convergence}
In the following, we prove the convergence theorem \ref{thm-main}. Before the proof, we establish  an estimate regarding the residuals $\mathcal R, \mathcal R_{1,2}$ for a $k$-approximate solution $u_a$. 
\begin{lemma}\label{lem-residual}
    Let $k\geq 1$ be any positive integer, $u_a$ be a $k$-approximate solution.  $\mathcal R=\mathcal R(u_a)$ is the residual of the fourth order equation introduced in \eqref{def-residual}, $\mathcal R_{1,2}=\mathcal R_{1,2}(u_a)$ are the residuals of the second order system, see \eqref{def-remainder-R12}. Then there exists a positive constant $C$ depending on system parameters such that  
    \begin{equation}
        \left<\mathcal R(u_a),u\right>_{L^2} \leq C\varep^{-4} \left(\varep \|\mathcal R_1\|_{L^2}+\|\mathcal R_2\|_{L^2}\right)\|u\|_{\Htwoin}. 
    \end{equation}
\end{lemma}
\begin{proof}
    The fourth order equation residual $\mathcal R$ and the second order equation residuals $\mathcal R_{1,2}$ has  the relation \eqref{relation-remainders}. Using the definition of the mass-projection operator $\Pi_0$, see \eqref{def-massfunc}, implies 
    \begin{equation}
    \begin{aligned}
         \left<\mathcal R(u_a),u\right>_{L^2} &=\frac{1}{\varep^3} \left<\mathcal R_1, u\right>_{L^2} -\frac{1}{\varep^4}\left<(\varep^2 \Delta -W''(u_a))\mathcal R_2, u\right>_{L^2}\\
         & \quad + \frac{1}{\varep^4|\Omega|} \int_\Omega \left(\varep^2 \Delta -W''(u_a)\right)\mathcal R_2 dx \int_\Omega u dx. 
         \end{aligned}
    \end{equation}
    Integrating by parts  yields 
    \begin{equation*}
    \begin{aligned}
         \left<\mathcal R(u_a),u\right>_{L^2} &=\frac{1}{\varep^3} \left<\mathcal R_1, u\right>_{L^2} -\frac{1}{\varep^4}\left<\mathcal R_2, (\varep^2\mathrm L_\varep)u\right>_{L^2} + \frac{1}{\varep^4|\Omega|} \int_\Omega \left( -W''(u_a)\right)\mathcal R_2 dx \int_\Omega u dx. 
         \end{aligned}
    \end{equation*}
    The Lemma follows by H\"older's inequality and $\|u\|_{L^1}\leq C\|u\|_{L^2}$ for a bounded domain. 
\end{proof}
Now we are in the position to prove the main theorem. 
\begin{proof}[Proof of Theorem \ref{thm-main}]
   The existence of the $k$-approximate solution, denoted as $u_a$, is given in Section \ref{sec-approximate-solutions}. In the following, we prove the convergence. Let $u=u_\varep-u_a$, where $u_\varep$ is the actual solution to \eqref{eq-FCH-L2}. The flow of $u$ is given in \eqref{flow-decomposition}. Taking the $L^2$-inner product on both sides of the equation \eqref{flow-decomposition} with $u$ yields 
    \begin{equation}
        \frac12 \frac{d}{dt} \|u\|_{L^2}^2 +\left<\Pi_0\mathbb L_\varep u, u \right>_{L^2} = \left<\mathcal R(u_a), u\right>_{L^2} +\left<\Pi_0\mathcal N(u), u\right>_{L^2}. 
    \end{equation}
From the definition of the $k$-approximate solution in Definition \ref{def-k-approximate sol}, the residuals $|\mathcal R_{1,2}|$ can be bounded by $C\varep^{k+1}$ uniformly. This combined with Lemma \ref{lem-residual} implies 
\begin{equation}\label{est-residual}
    \left<\mathcal R(u_a), u\right>_{L^2}  \leq C\varep^{k-3} \|u\|_{\Htwoin}. 
\end{equation}
Applying Cauchy-Schwartz inequality, the  nonlinear estimates in Lemma \ref{lem-nonlinear-est}, with $\delta_*=\frac14$ there, and the residual estimate \eqref{est-residual} implies 
    \begin{equation}\label{est-main-theorem-1}
        \frac12\frac{d}{dt} \|u\|_{L^2}^2 +  \left<\Pi_0\mathbb L_\varep u, u\right>_{L^2}\leq \frac14 \left<\mathbb L_\varep u, u\right>_{L^2}  +\delta_0\|u\|_{\Htwoin}^2 +\frac{C}{\varep^4}\|u\|_{L^2}^{2+r_0}+\frac{C}{\varep^{10}}\|u\|_{L^2}^4+C\varep^{2k-6}. 
    \end{equation}
   Here the constant $C$ might depend on $\delta_0$.  
   From the decomposition $u={w}+u^\parallel$, using the triangle inequality and inequality \eqref{est-u-parallel-Htwoin} implies 
   \begin{equation}
       \|u\|_{\Htwoin}^2 \leq  \| {w}\|_{\Htwoin}^2+\varep^4\|Z\|_{H^2(\Gamma)}^2+\|Z\|_{L^2(\Gamma)}^2. 
   \end{equation} 
   Using the estimate in Lemma \ref{lem-Z-u-L2} and the first  coercivity inequality in Theorem \ref{thm-coercivity} we derive 
   \begin{equation}
       \|u\|_{\Htwoin}^2 \leq C\varep^4 \left<\mathbb L_\varep u, u\right>_{L^2}+C\|u\|_{L^2}^2.
   \end{equation} 
   With this estimate, if we choose $\varep$ small enough depending on system parameters  only,  the inequality in \eqref{est-main-theorem-1} becomes
   \begin{equation}
       \frac12\frac{d}{dt} \|u\|_{L^2}^2 +  \left<\Pi_0\mathbb L_\varep u, u\right>_{L^2}\leq \frac13 \left<\mathbb L_\varep u, u\right>_{L^2}  +C \|u\|_{L^2}^2  +\frac{C}{\varep^4}\|u\|_{L^2}^{2+r_0}+\frac{C}{\varep^{10}}\|u\|_{L^2}^4+C\varep^{2k-6}. 
   \end{equation}
   From the second inequality in the coercivity Theorem \ref{thm-coercivity} and Young's inequality, we deduce 
    \begin{equation}
         \frac{d}{dt} \|u\|_{L^2}^2 \leq C \|u\|_{L^2}^2  +\frac{C}{\varep^4}\|u\|_{L^2}^{2+r_0}+\frac{C}{\varep^{10}}\|u\|_{L^2}^4 +\frac{C}{\varep^6}\left(\int_\Omega udx\right)^2+ C \varep^{2k-6} .
    \end{equation} 
    From the mass constraint, we have 
    \begin{equation}\label{est-u-L2-0}
         \frac{d}{dt} \|u\|_{L^2}^2 \leq C \|u\|_{L^2}^2  +\frac{C}{\varep^4}\|u\|_{L^2}^{2+r_0}+\frac{C}{\varep^{10}}\|u\|_{L^2}^4+ C \varep^{2k-6} .
    \end{equation} 
    Let $T_*\leq T$  be the largest  positive constant such that $\|u\|_{L^2}^2\leq \varep^{10}\wedge \varep^{2n}$ for all $t\in [0, T_*)$, that is
    \begin{equation}
       T_*:= \max_{T_0\leq T}\left\{T_0:  \|u(\cdot,t)\|_{L^2}^2\leq  \varep^{2n}\wedge \varep^{10}, \forall t\in [0, T_0)\right\}. 
    \end{equation}
We claim that $T_*=T$ for some suitably chosen $K$. In fact, from \eqref{est-u-L2-0} and the definition of $T_*$ it holds that
\begin{equation}
    \frac{d}{dt}\|u\|_{L^2}^2 \leq C\|u\|_{L^2}^2 +C\varep^{2k-6}, \quad \forall t\in [0, T_*). 
\end{equation}
Multiplying both sides by $e^{-Ct}$ and integrating implies for any $t\in[0, T_*)$
\begin{equation}
\begin{aligned} 
    \|u(\cdot,t)\|_{L^2}^2 &\leq e^{Ct}\|u(\cdot,0)\|_{L^2}^2+Ce^{Ct} \varep^{2k-6}\\
    &\leq e^{Ct} \left(C_0+C\varep^{2k-6-2n}\right)\varep^{2n}\wedge  e^{Ct} \left(C_0+C\varep^{2k-16}\right)\varep^{10}. 
    \end{aligned} 
\end{equation}
For $2k>(6+2n)\wedge 16$ and taking $C_0,\varep$ small enough depending on $T$ so that 
\begin{equation}
    e^{CT} \left(C_0+C\varep^{2k-6-2n}\right) +e^{CT} \left(C_0+C\varep^{2k-16}\right)<\frac12.
\end{equation}
The inequality becomes 
\begin{equation}
    \|u(\cdot,t)\|_{L^2}^2 \leq \frac12 \varep^{2n}\wedge \varep^{10}. 
\end{equation}
 Hence $T=T_*$, this completes the proof.  

\end{proof}

\section{Linear coercivity}\label{sec-linear-coercivity}
In this section, we study the linearized operator, $\mathbb L_\varep$ in \eqref{form-bL}, of the system around a $k$-approximate solution $u_a$ and establish the coercivity Theorem \ref{thm-coercivity}. We mention that this section is comparable to the sections 4-5 in \cite{fei2021phase}. 

The proof of coercivity in Theorem \ref{thm-coercivity} relies on the analysis of the linearized operator  $\mathbb L_\varepsilon$  by expanding it in local coordinates near the  interface  $\Gamma$. 
 Precisely, using the expansion of the Laplacian, in \eqref{expansion-Laplace}, and the potential $ W(u)$, we express the linear operator  $\mathrm L_\varepsilon$  in terms of tangential and normal components. Specifically: 
    \begin{equation}\label{expansion-rL}
        \varep^2 \mathrm L_\varep =\varep^2 \Delta_\Gamma -\mathrm L_0 +\varep\Big( {H}(s)  +\varep zH_2(s)\Big)\p_z-\varep^2 W'''(\phi_0)u_2 +\varep^3 \tilde{\mathrm D}_z. 
    \end{equation}
    where $\tilde{\mathrm D}_z$ is a higher-order correction term involving geometric factors from  $\mathrm D_z$  and other small corrections, that is, 
    \begin{equation}\label{def-tilde-Dz}
        \tilde{\mathrm D}_z= \mathrm D_z +\varep^{-3}\left(-W''(u_a)+W''(\phi_0)+\varep^2W'''(\phi_0)u_2\right). 
    \end{equation}
    In comparison with \eqref{est-rLu-Htwoin}, when  $u$  is replaced by its parallel component  $u^\parallel$, which is aligned with the localized function  $\varphi(z, s) $, we obtain a better (smaller) upper bound for  $\|\varepsilon^2 \mathrm{L}_\varepsilon u^\parallel\|_{L^2}$  by leveraging the expansion in \eqref{expansion-rL}. This improvement arises from the specific structure and localization properties of  $u^\parallel $, which allow more precise control over the associated terms in the expansion.

\begin{lemma}\label{lem-est-rR-u-parallel}
   Let $u^\parallel$ be given as in \eqref{orthogonal-decomposition-Z} for some $Z=Z(s)$, then there exists a constant $C$   such that 
    \begin{equation}
        \|\varep^2\mathrm L_\varep u^\parallel\|_{L^2}\leq C\varep^2\|Z\|_{H^2(\Gamma)}+C\varep\|Z\|_{L^2(\Gamma)}. 
    \end{equation}
    Moreover, for some positive constants $C_1, C_2$ it holds that 
    \begin{equation}
\begin{aligned} 
      \int_\Omega  |\varep^2\mathrm L_\varep u^\parallel|^2 dx &\geq  \varep^2 \int_{-\ell/\varep}^{\ell/\varep}  |\phi_0''|^2 dz\|{H}Z\|_{L^2(\Gamma)}^2+C_1\varep^4 \| Z\|_{H^2(\Gamma)}^2  -C_2\varep^4 \|Z\|_{L^2(\Gamma)}^2. 
      \end{aligned} 
\end{equation}
Here constants shown depend on system parameters only, particularly they are independent of $\varep,\delta$. 
\end{lemma}
\begin{proof}
Recall that  $u^\parallel = Z(s) \varphi(z,s)$, where  $\varphi(z,s)$  is localized near the interface  $\Gamma$  and  $Z(s)$  is the projection.
 Applying the expansion of the operator $\varep^2 \mathrm L_\varep$ in \eqref{expansion-rL} to  $u^\parallel$ we get   
\begin{equation}\label{def-cR_L}
\begin{aligned} 
    \sqrt\varep \varep^2 \mathrm L_\varep u^\parallel  =  & \phi_0' \varep^2 \Delta_\Gamma Z +\varep {H}(s) \phi_0'' Z +\varep^2\mathscr R_{\mathrm L}.
    \end{aligned}
\end{equation}
where  $\phi_0^{\prime}$  and  $\phi_0^{\prime\prime}$  are derivatives of the homoclinic solution  $\phi_0$, and  $\mathscr R_{\mathrm L}$  is the remainder term defined as:
\begin{equation}\label{def-cRL}
    \mathscr {R}_{\mathrm L} :=  \left(H_2(s)z\phi_0'' -W'''(\phi_0)u_2\phi_0'\right)Z(s)  +\varep\tilde{\mathrm D}_z(\phi_0'Z).
\end{equation}
Note that $\Gamma$ is smooth. Using Lemma \ref{lem-auxilliary-operator-Gamma-Gamma0}, we  control the  $L^2$-norm of  $\mathscr R_{\mathrm L}$  as:  
\begin{equation}\label{est-cRL} 
  \frac{1}{\varep} \int_\Omega  |\mathscr R_{\mathrm L}|^2 dx +\int_{-\ell/\varep}^{\ell/\varep} \int_\Gamma \mathscr R_{\mathrm L}^2 dsdz \leq C\left(\|Z\|_{L^2(\Gamma)}^2+\varep^4\|Z\|_{H^2(\Gamma)}^2\right).
\end{equation}
We consider the  $L^2$-norm of  $\varepsilon^2 \mathrm L_\varepsilon u^\parallel$. From \eqref{def-cR_L} expanding the square implies 
\begin{equation}\label{est-rL-u-parallel} 
      \int_\Omega  |\varep^2\mathrm L_\varep u^\parallel|^2 dx =\mathcal I_1+\mathcal I_2+\mathcal I_3, 
\end{equation}
where the terms on the right hand side are defined as 
\begin{equation}
    \begin{aligned}
        \mathcal I_1 & := \frac{1}{\varep}\int_\Omega \left|\phi_0' \varep^2 \Delta_\Gamma Z +\varep {H}(s) \phi_0'' Z \right|^2 dx; \\
        \mathcal I_2 &:= \frac{2\varep^2}{\varep} \int_\Omega \left(\phi_0' \varep^2 \Delta_\Gamma Z +\varep {H}(s) \phi_0'' Z \right) \mathscr R_{\mathrm L} dx; \\
        \mathcal I_3 &:=  \frac{\varep^4}{\varep} \int_\Omega  |\mathscr R_{\mathrm L}|^2 dx.
    \end{aligned}
\end{equation}
The third item $\mathcal I_3$ is bounded by $C\varep^4\|Z\|_{H^2(\Gamma)}^2$ from \eqref{est-cRL}. The item $\mathcal I_1$ can be bounded by terms involving  $\| Z \|_{H^2(\Gamma)}$  via Cauchy-Schwartz inequality, precisely 
\begin{equation}
    |\mathcal I_1| \leq C\varep^4\|Z\|_{H^2(\Gamma)}^2+\varep^2\|{H}\|_{L^\infty}^2\|Z\|_{L^2(\Gamma)}^2. 
\end{equation}
The item $\mathcal I_2$ can be bounded by $\mathcal I_1+\mathcal I_3$ from Cauchy-Schwartz inequality. Thus, we obtain:
\begin{equation} 
\int_\Omega | \varepsilon^2 \mathrm L_\varepsilon u^\parallel |^2 dx \leq C\varep^4\|Z\|_{H^2(\Gamma)}^2+\varep^2\|{H}\|_{L^\infty}^2\|Z\|_{L^2(\Gamma)}^2.
\end{equation}

To establish the lower bound, we analyze  $\mathcal I_1$  in detail. Using the second inequality of the Jacobian in \eqref{est-Jacobian} implies 
\begin{equation}
\begin{aligned} 
    \mathcal I_1 \geq & \int_{-\ell/\varep}^{\ell/\varep} \int_\Gamma |\phi_0'\varep^2 \Delta_\Gamma Z+\varep {H}(s)\phi_0'' Z|^2(1+\varep z{H} ) dsdz \\
    &-C\varep^2  \int_{-\ell/\varep}^{\ell/\varep} \int_\Gamma |\phi_0'\varep^2 \Delta_\Gamma Z+\varep {H}(s)\phi_0'' Z|^2z^2  dsdz. 
    \end{aligned}
\end{equation}
The second term is bounded from Cauchy-Schwartz inequality, exponential day of $\phi_0',\phi_0''$ and Lemma \ref{lem-auxilliary-operator-Gamma-Gamma0}, particularly,
\begin{equation}
\begin{aligned} 
    \mathcal I_1 \geq & \int_{-\ell/\varep}^{\ell/\varep} \int_\Gamma |\phi_0'\varep^2 \Delta_\Gamma Z+\varep {H}(s)\phi_0'' Z|^2(1+\varep z{H} ) dsdz-C\varep^6\|Z\|_{H^2(\Gamma)}^2 -C\varep^4\|Z\|_{L^2(\Gamma)}^2. 
    \end{aligned}
\end{equation}
We expand the square in the first item on the right hand side  and decompose it into three terms: 
\begin{equation}
    \int_{-\ell/\varep}^{\ell/\varep} \int_\Gamma |\phi_0'\varep^2 \Delta_\Gamma Z+\varep {H}(s)\phi_0'' Z|^2(1+\varep z{H} ) dsdz = \mathcal I_{11}+\mathcal I_{12}+\mathcal I_{13}, 
\end{equation}
where the terms  $\mathcal I_{11}$, $\mathcal I_{12}$,  and  $\mathcal I_{13}$  are defined as follows:
\begin{equation}
\begin{aligned}
    &\mathcal I_{11} := \varep^4 \int_{-\ell/\varep}^{\ell/\varep} \int_\Gamma |\phi_0'|^2 |\Delta_\Gamma Z|^2(1+\varep z{H}) dsdz;\\
    &\mathcal I_{12}:=2\varep^3\int_{-\ell/\varep}^{\ell/\varep} \int_\Gamma \phi_0'\phi_0'' {H}(s)\Delta_\Gamma ZZ(1+\varep z{H})dsdz;\\
    &\mathcal I_{13} := \varep^2 \int_{-\ell/\varep}^{\ell/\varep}  |\phi_0''|^2dz\|{H}Z\|_{L^2(\Gamma)}^2.
    \end{aligned}
\end{equation}
Here we also used that $z|\phi_0''|^2$ is odd with respect to the $z$-variable so that the integral on any symmetric domain of $z$ is zero. Since  $\phi_0^{\prime} $ is localized and $ \| \phi_0^{\prime} \|_{L^2(\mathbb R)} = m_1 $, applying Lemma \ref{lem-auxilliary-operator-Gamma-Gamma0} in the appendix,  the first term has the following lower bound 
\begin{equation}\label{est-rL-u-parallel-CI11}
    \mathcal I_{11}\geq \frac{m_1^2\varep^4}{2} \|Z\|_{H^2(\Gamma)} - C\varep^4\|Z\|_{L^2(\Gamma)}.
\end{equation}
 Applying the Cauchy–Schwarz inequality and using properties of  $\phi_0^{\prime}$  and  $\phi_0^{\prime\prime}$ : 
\begin{equation}\label{est-rL-u-parallel-CI12}
    |\mathcal I_{12}| \leq C\varep^4 \|Z\|_{H^2(\Gamma)}\|Z\|_{L^2(\Gamma)}\leq \frac{m_1^2\varep^4}{4} \|Z\|_{H^2(\Gamma)}+C\varep^4\|Z\|_{L^2(\Gamma)}^2. 
\end{equation}
Summing the estimates for  $\mathcal I_{11}$ (estimate  \eqref{est-rL-u-parallel-CI11}),  $\mathcal I_{12}$(estimate \eqref{est-rL-u-parallel-CI12}), with the form of  $\mathcal I_{13}$: 
\begin{equation}\label{kernel-est-positivesquare}
    \mathcal I_1\geq \varep^2 \int_{-\ell/\varep}^{\ell/\varep}  |\phi_0''|^2dz\|{H}Z\|_{L^2(\Gamma)}^2 +\frac{m_1^2\varep^4}{4} \|Z\|_{H^2(\Gamma)} -C\varep^4 \|Z\|_{L^2(\Gamma)}^2. 
\end{equation}
Again using $dx=\varep J(\varep z, s)dsdz$ with the Jacobian satisfying the first bound in \eqref{est-Jacobian}, and the exponential decay of $\phi_0',\phi_0''$ yields 
\begin{equation}\label{est-rL-u-parallel-CI2-0} 
\begin{aligned} 
    \mathcal I_2 & \geq 2\varep^2\int_{-\ell/\varep}^{\ell/\varep} \int_\Gamma (\phi_0'\varep^2 \Delta_\Gamma Z+\varep H \phi_0''Z)\mathscr R_{\mathrm L} dsdz - C\varep^3 \int_{-\ell/\varep}^{\ell/\varep} \int_\Gamma (|\varep^2 \Delta_\Gamma Z|+\varep |Z|)|\mathscr R_{\mathrm L}| e^{-\nu|z|} dsdz \\
    &\geq  2\varep^3\int_{-\ell/\varep}^{\ell/\varep} \int_\Gamma  H \phi_0''Z \mathscr R_{\mathrm L} dsdz - C\varep^4 \int_{-\ell/\varep}^{\ell/\varep} \int_\Gamma (| \Delta_\Gamma Z|+ |Z|)|\mathscr R_{\mathrm L}|e^{-\nu|z|} dsdz
    \end{aligned}
    \end{equation} 
Here to get the second inequality we have absorbed the first item into the the remainder. Applying  Cauchy-Schwartz inequality gives
\begin{equation}
\begin{aligned} 
&\int_{-\ell/\varepsilon}^{\ell/\varepsilon} \int_\Gamma \left( |\Delta_\Gamma Z| + |Z| \right)e^{-\nu|z|} |\mathscr R_{\mathrm L}| \, ds \, dz \\
&\leq \left( \int_{-\ell/\varepsilon}^{\ell/\varepsilon} \int_\Gamma \left( |\Delta_\Gamma Z| +  |Z| \right)^2e^{-2\nu|z|} ds dz \right)^{1/2}
\left( \int_{-\ell/\varepsilon}^{\ell/\varepsilon} \int_\Gamma |\mathscr R_L|^2 ds dz \right)^{1/2}.
\end{aligned} 
\end{equation} 
Note that  $\Delta_\Gamma Z$  involves second derivatives of  $Z$,  we can control these terms using the  $H^2(\Gamma)$-norm of  $Z$ with the aid of Lemma \ref{lem-auxilliary-operator-Gamma-Gamma0} ; the second factor is bounded by \eqref{est-cRL}. Combining the two estimates, we get the following estimate 
\begin{equation}
C \varepsilon^4 \int_{-\ell/\varepsilon}^{\ell/\varepsilon} \int_\Gamma \left( |\Delta_\Gamma Z| + |Z| \right)e^{-\nu|z|} |\mathscr R_{\mathrm L}| \, ds \, dz
\leq C \varepsilon^4 \| Z \|_{H^2(\Gamma)}\left(\|Z\|_{L^2(\Gamma)} +\varep^2\|Z\|_{H^2(\Gamma)}\right).
\end{equation} 
Applying Cauchy-Schwartz inequality and taking $\varep$ small enough depending on system parameters only implies 
\begin{equation}
    C \varepsilon^4 \int_{-\ell/\varepsilon}^{\ell/\varepsilon} \int_\Gamma \left( |\Delta_\Gamma Z| + |Z| \right)e^{-\nu|z|} |\mathscr R_{\mathrm L}| \, ds \, dz
\leq \frac{m_1^2\varep^4}{16}\|Z\|_{H^2(\Gamma)}^2 +C\varep^4\|Z\|_{L^2(\Gamma)}^2. 
\end{equation}
This together with inequality \eqref{est-rL-u-parallel-CI2-0} implies 
    \begin{equation}\label{est-rL-u-parallel-CI2-1}
    \begin{aligned} 
    \mathcal I_2 &\geq 2\varep^3 \int_{-\ell/\varep}^{\ell/\varep} \int_\Gamma  H \phi_0''Z\mathscr R_{\mathrm L} dsdz -\frac{m_1^2\varep^4}{16} \|Z\|_{H^2(\Gamma)}^2 -C\varep^4 \|Z\|_{L^2(\Gamma)}^2.  
    \end{aligned}
\end{equation}
From the definition of $\mathscr R_{\mathrm L}$, \eqref{def-cRL}, using the odd-even parity of $\phi_0''$ and  $z\phi_0'', \phi_0'$ cancels the leading order and 
\begin{equation}
\begin{aligned} 
    2\varep^3 \int_{-\ell/\varep}^{\ell/\varep} \int_\Gamma  H \phi_0''Z\mathscr R_{\mathrm L} dsdz & \geq - C \varep^4 \int_{-\ell/\varep}^{\ell/\varep} \int_\Gamma |Z| \left(|\tilde{\mathrm D}_z(\phi_0'Z)|\right) dsdz.
    \end{aligned} 
\end{equation}
In light of the definition of $\tilde{\mathrm D}_z$, which is a differential operator with bounded coefficients involving only the first derivative with respect to the  $z$-variable, we can apply the Cauchy-Schwarz inequality and use Lemma \ref{lem-auxilliary-operator-Gamma-Gamma0} to derive 
\begin{equation}
\begin{aligned} 
    2\varep^3 \int_{-\ell/\varep}^{\ell/\varep} \int_\Gamma  H \phi_0''Z\mathscr R_{\mathrm L} dsdz & \geq  -C\varep^4\|Z\|_{L^2(\Gamma)}^2 -C \varep^6\| Z\|_{H^2(\Gamma)}^2. 
    \end{aligned} 
\end{equation}
Combining the estimate above with \eqref{est-rL-u-parallel-CI2-1} and choosing $\varep$ small enough implies  
\begin{equation}\label{est-rL-u-parallel-CI2}
    \mathcal I_2\geq -\frac{m_1^2\varep^4}{8} \| Z\|_{H^2(\Gamma)}^2 -C\varep^4\|Z\|_{L^2(\Gamma)}^2. 
\end{equation}
Note that $\mathcal I_3\geq 0$, as it represents the square of a norm. Summing the estimates for $\mathcal I_1$ from equation \eqref{kernel-est-positivesquare} and $\mathcal I_2$ from equation \eqref{est-rL-u-parallel-CI2} gives us the following lower bound 
\begin{equation}
\begin{aligned} 
      \int_\Omega  |\varep^2\mathrm L_\varep u^\parallel|^2 dx &\geq  \int_{-\ell/\varep}^{\ell/\varep}  |\phi_0''|^2 dz\|{H}Z\|_{L^2(\Gamma)}^2+\frac{m_1^2\varep^4}{8} \| Z\|_{H^2(\Gamma)}^2  -C\varep^4 \|Z\|_{L^2(\Gamma)}^2.
      \end{aligned} 
\end{equation}
This inequality provides the desired lower bound for the norm  $\| \varepsilon^2 \mathrm L_\varepsilon u^\parallel \|_{L^2}$  by taking $C_1=\frac{m_1^2\varep^4}{8}$ and $C_2=C$ as defined in the above estimates. The proof is complete.

\end{proof}

Now, we proceed with the proof of the linear coercivity in Theorem \ref{thm-coercivity}, utilizing the results of Propositions \ref{prop-coercivity-kernel-uparallel}, \ref{prop-est-cross-terms}, and \ref{prop-est-perpendicular-terms}, which will be established in subsequent sections.  
\begin{proof}[Proof of Theorem \ref{thm-coercivity}]
For a solution remainder $u$, we decompose it as $u=u^\parallel+{w}$, see \eqref{orthogonal-decomposition-Z}. Under this decomposition, the following holds: 
\begin{equation}
    \left<\mathbb L_\varep u, u\right>_{L^2}= \left<\mathbb L_\varep u^\parallel, u^\parallel\right>_{L^2}+2\left<\mathbb L_\varep u^\parallel, {w}\right>_{L^2}+\left<\mathbb L_\varep {w}, {w}\right>_{L^2}. 
\end{equation}
The results of Propositions \ref{prop-coercivity-kernel-uparallel},   \ref{prop-est-perpendicular-terms} and \ref{prop-est-cross-terms}, allow us to conclude that
 \begin{equation}
        \varep^4 \left<\mathbb L_\varep u, u\right>_{L^2}\geq C_1\varep^2 \|{w}\|_{\Htwoin}^2+C_1\varep^4 \|Z\|_{H^2(\Gamma)}^2 -C_2\varep^4\|Z\|_{L^2(\Gamma)}^2 
    \end{equation} 
   for some positive constants $C_1, C_2$ independent of $\varep$. The first inequality then follows by dividing both sides by $\varep^4$ with the aid of Lemma \ref{lem-Z-u-L2}. The constant $C_2$ is updated during the derivation. 

In addition, the definition of the mass projection operator $\Pi_0$ implies 
\begin{equation}\label{coercivity-Pi0-bL-0}
     \varep^4 \left<\Pi_0\mathbb L_\varep u, u\right>_{L^2} = \varep^4 \left<\mathbb L_\varep u, u\right>_{L^2} -\frac{\varep^4}{|\Omega|} \int_\Omega \mathbb L_\varep u dx \int_\Omega udx. 
\end{equation}
Recalling the form of $\varep^4\mathbb L_\varep$ given in \eqref{form-bL}, using integration by parts yields
\begin{equation}
    \varep^4 \int_\Omega \mathbb L_\varep u dx = \int_\Omega \varep^2 \mathrm L_\varep u\,  \varep^2 \mathrm L_\varep 1 dx+\varep \int_\Omega \mathrm R_a u dx. 
\end{equation}
For a smooth $k$-approximate solution $u_a$, the quantities $|\varep^2 \mathrm L_\varep 1|$ and $|\mathrm R_a|$ are uniformly bounded. Since the domain $\Omega$ is bounded,  the $L^1$-norm can be controlled by the $L^2$-norm. Combining these yields 
\begin{equation}
    \left|\int_\Omega \varep^4 \mathbb L_\varep u dx \right| \leq C\|\varep^2 \mathrm L_\varep u\|_{L^2} +C\varep\|u\|_{L^2}. 
\end{equation}
From the decomposition $u=u^\parallel+{w}$, applying triangle inequality, Lemma \ref{lem-est-rR-u-parallel} and estimate \eqref{est-rLu-Htwoin} gives
\begin{equation}
\begin{aligned}
    \|\varep^2 \mathrm L_\varep u\|_{L^2} &\leq \|\varep^2 \mathrm L_\varep {w}\|_{L^2} + \|\varep^2 \mathrm L_\varep u^\parallel\|_{\Htwoin} \\
    &\leq C\|{w}\|_{\Htwoin} +C \varep \|Z\|_{L^2}.
    \end{aligned}
\end{equation}
Combining the previous two inequalities yields 
\begin{equation}
    \left|\int_\Omega \varep^4 \mathbb L_\varep u dx \right| \leq C\|{w}\|_{\Htwoin} +C \varep \|Z\|_{L^2(\Gamma)}. 
\end{equation}
Now, returning to identity \eqref{coercivity-Pi0-bL-0} we obtain  
\begin{equation}
    \varep^4 \left<\Pi_0\mathbb L_\varep u, u\right>_{L^2}  \geq  \varep^4 \left<\mathbb L_\varep u, u\right>_{L^2} -C(\|{w}\|_{\Htwoin} + \varep \|Z\|_{L^2(\Gamma)}) \left|\int_\Omega udx\right| .
\end{equation}
Note that $\|Z\|_{L^2(\Gamma)}\leq C\|u\|_{L^2}$ from Lemma \ref{lem-Z-u-L2}. Using the Cauchy-Schwartz inequality, we deduce that for any  $\delta_* > 0$, there exists a positive constant  $C$  such that:
\begin{equation}
    \varep^4 \left<\Pi_0\mathbb L_\varep u, u\right>_{L^2}  \geq  \varep^4 \left<\mathbb L_\varep u, u\right>_{L^2} -\delta_* \varep^2 \|{w}\|_{\Htwoin}^2 -\delta_*\varep^4 \|u\|_{L^2(\Gamma)}^2 -\frac{C}{\delta_*\varep^2}\left(\int_\Omega udx\right)^2 .
\end{equation}
By taking $\delta_*$ sufficiently small, the second coercivity inequality in the Theorem follows from the inequality above and the first coercivity inequality in the Lemma. This completes the proof.   
\end{proof}

\subsection{Kernel estimates}\label{sec-linear-coercivity-kernel-est}
In this section, we establish the kernel estimates, specifically the coercivity of  the linearized operator $\mathbb L $ restricted to the space spanned by  $\varphi$, see \eqref{def-varphi}. This involves obtaining a lower bound for the term  $\langle \mathbb L_\varep u^\parallel, u^\parallel \rangle_{L^2}$ in terms of the projection  $Z=Z(s)$. The main result in this section is stated below. 

\begin{prop}\label{prop-coercivity-kernel-uparallel} 
 For any $u^\parallel=\varphi(z) Z(s)$ with $\sqrt{\varep}\varphi(z)=\phi_0'(z)$, there exist positive constants $C_1, C_2$, depending only on system parameters and $\varphi$,  such that 
    \begin{equation}
    \begin{aligned}
      \varep^4\left<\mathbb L_\varep u^\parallel, u^\parallel\right>_{L^2}
       &\geq C_1\varep^4\|Z\|_{H^2(\Gamma)}^2-C_2{\varep^4} \|Z\|_{L^2(\Gamma)}^2. 
       \end{aligned} 
    \end{equation}
\end{prop}
\begin{proof}
    By the form of $\mathbb L_\varep$ in \eqref{form-bL}, we have 
    \begin{equation}\label{est-kernel-expansion}
        \varep^4 \left<\mathbb L_\varep u^\parallel, u^\parallel\right>_{L^2}   = \varep^4\left<\mrL_\varep u^\parallel, \mrL_\varep u^\parallel\right>_{L^2} +\varep\left<\mathrm R_a u^\parallel, u^\parallel\right>_{L^2}. 
    \end{equation}
Here we recall that  $\mathrm R_a$ is defined in \eqref{def-Ra}. We need to control the terms on the right hand side one by one. The first term has already been addressed in Lemma \ref{lem-est-rR-u-parallel}, which provides a positive leading order term. For clarity of presentation, we postpone addressing the last term to Lemma \ref{lem-est-rR-u-parallel-square}. This lemma shows that the last term has a leading order that is negative. Specifically, by combining the results of these two lemmas, we see that the leading orders cancel out, the proposition follows.
\end{proof}

\begin{lemma}\label{lem-est-rR-u-parallel-square}Under the same assumptions as in Proposition \ref{prop-coercivity-kernel-uparallel}, there exists  a positive constant $C$ such that 
    \begin{equation}\label{kernel-est-remainder}
    \varep\left<\mathrm R_a u^\parallel, u^\parallel \right>_{L^2} \geq  -  \varep^2\int_{-\ell/\varep}^{\ell/\varep}  |\phi_0''|^2 dz\|{H}Z\|_{L^2(\Gamma)}^2 -C\varep^4\|Z\|_{L^2(\Gamma)}^2. 
\end{equation} 
\end{lemma}
\begin{proof}
   From the definition of $\mathrm R_a$,  expansion of $u_a,v_a$ and \eqref{approximate-solution-form-0} yields
\begin{equation}\label{form-rRa-1}
\begin{aligned} 
    \mathrm R_a &= -W'''(\phi_0)\phi_0' \left(\Delta d_0 +\varep \Delta d_1\right) +W'''(\phi_0)\varep D_0 z\phi_0'\\
    & \quad  -\varep^2 \left(u_2W^{(4)}(\phi_0)\phi_0'\Delta d_0 +W'''(\phi_0) v_2\right) +\varep^3\mathrm R_{\geq 3}. 
    \end{aligned}
\end{equation}
Here and below the remainder $\mathrm R_{\geq 3}$ denotes a term that may vary line by line but is uniformly bounded, i.e.  $|\mathrm R_{\geq 3}|\lesssim 1$. 
In the following we address on the approximations of some geometric quantities shown above. For $r=\sum_{k\geq 0} \varep^k d_k$, $\Delta d_0+\varep \Delta d_1=\Delta r-\varep^2 \Delta d_2+O(\varep^3)$ and we deduce  from the first expansion in \eqref{expansion--metric-derivative} that  
\begin{equation}\label{approximation-d0+d1}
\begin{aligned} 
\Delta d_0+\varep \Delta d_1
&= {H}(s) +\varep zH_2(s) +\varep^2 z^2 H_3(s) -\varep^2\Delta d_2 +O(\varep^3).  
\end{aligned}
\end{equation}
In addition, from the definition  of $D_k$ in \eqref{def-Dk}, we have $D_0=\nabla \Delta r\cdot \nabla r+\frac12 |\Delta r|^2- \sum_{k=1} \varep^{k}D_k$. Using the expansions in \eqref{expansion--metric-derivative} gives
\begin{equation}\label{approximation-D0}
\begin{aligned} 
    D_0 
    & = H_2(s) +\frac12 {H}^2(s) +\varep z(2H_3(s)+H_2(s){H}(s))-\varep D_1+O(\varep^2). 
    \end{aligned} 
\end{equation}
Plugging the relations in \eqref{approximation-d0+d1}-\eqref{approximation-D0} into the right hand side of \eqref{form-rRa-1}, we deduce $\mathrm R_a= \sum_{k=0}^2\varep^k{\mathrm R}_k+\varep^3 \mathrm R_{\geq 3}$ with 
       $ {\mathrm R}_0 := -W'''(\phi_0) \phi_0' {H}(s), 
        {\mathrm R}_1:= \frac12W'''(\phi_0)z\phi_0' {H}^2(s) ;$
        and the $\varep^2$-order $\mathrm R_2$ is given by  
        \begin{equation}
        \begin{aligned} 
        {\mathrm R}_2 & := -W'''(\phi_0)z^2\phi_0'\left(-H_3(s) -{H}(s)H_2(s)\right)+W'''(\phi_0)\phi_0' \Delta d_2  -W^{(4)}(\phi_0) \phi_2\phi_0' {H}(s) \\
        &
        \quad -W'''(\phi_0)(z\phi_0' D_1+v_2). 
    \end{aligned}
\end{equation}
Recall that the Jacobian $J(r,s) = 1+\varep z{H}(s) +\varep^2 z^2 \kappa_2(s) +z^3 O(\varep^3)$, as stated in \eqref{expansion-Jacobian}. We further expand 
\begin{equation}\label{expansion-remainder-jacobian}
    \varep \mathrm R_a J(r,s)=\varep \mathrm I_1 +\varep^2 \mathrm I_2 +\varep^3\mathrm I_3+O(\varep^4)(1+|z|^3), 
\end{equation}
where $\mathrm I_{1,2}$ are given by:  
\begin{equation}
    \begin{aligned}
        \mathrm I_1 &:= -W'''(\phi_0)\phi_0' {H}(s); \qquad \qquad \\
        \mathrm I_2& := -\frac12W'''(\phi_0) z\phi_0'{H}^2(s) ; \end{aligned} 
\end{equation}
and the third term $\mathrm I_3$ has the form: 
\begin{equation}
 \begin{aligned} 
         \mathrm I_3 & := \frac12W'''(\phi_0)z^2 \phi_0' {H}(s) \left({H}^2(s)-2\kappa_2(s)\right)   -W'''(\phi_0)z^2\phi_0'\Big(-H_3(s)-{H}(s)H_2(s)\Big) \\
         &\quad+W'''(\phi_0)\phi_0'\Delta d_2-W^{(4)}(\phi_0)u_2\phi_0'{H}(s)-W'''(\phi_0)(z\phi_0' D_1+v_2). 
    \end{aligned} 
\end{equation}
Note from Lemma \ref{lem-L0}, we have $\mathrm L_0\phi_0''=-W'''(\phi_0)|\phi_0'|^2$. As $\phi_0''$ is odd, so does the function $\mathrm L_0\phi_0''$ and $W'''(\phi_0)|\phi_0'|^2$. This implies $W'''(\phi_0)$ is odd, and  the terms in $\mathrm I_3$ are all odd functions with respect  to the $z$-variable, except the last term $-W'''(\phi_0)(z\phi_0' D_1+v_2)$.  Noting that the localized function $\sqrt{\varep} u^\parallel =\phi_0'Z$, we have $\varep  |u^\parallel|^2= |\phi_0'|^2 Z^2$ and 
\begin{equation}\label{kernel-est-remainder-1}
    \begin{aligned}
        \varep\left<\mathrm R_a u^\parallel, u^\parallel \right>_{L^2}   
        \geq   \int_{-\ell/\varep}^{\ell/\varep} \int_{\Gamma} \left(\varep \mathrm K_1 +\varep^2 \mathrm K_2 +\varep^3 \mathrm K_3 \right) dsdz-C\varep^4 \|Z\|_{L^2(\Gamma)}^2, 
    \end{aligned}
\end{equation}
where $\mathrm K_l:=\mathrm I_l|\phi_0'|^2Z^2$ for $l=1,2,3$. 
  We  claim that the integral involving $\mathrm K_1$ and $\mathrm K_3$ is also zero, that is, 
\begin{equation}\label{claim-K3-integration-zero}
    \int_{-\ell/\varep}^{\ell/\varep} \int_{\Gamma}  (\varep \mathrm K_1+\varep^3\mathrm K_3) dsdz=0.
\end{equation}
Using the fact that $W'''(\phi_0)$ is odd, 
it's straightforward to see that the integral of $\mathrm K_1$ is zero since $\mathrm K_1$ is also an odd function with respect to the $z$-variable. It remains to show the integral of $\mathrm K_3$ is also zero.  
In fact, using the odd-even parity, we have
\begin{equation}
    \int_{-\ell/\varep}^{\ell/\varep} \int_{\Gamma}  \mathrm K_3 dsdz = -\int_{-\ell/\varep}^{\ell/\varep} \int_{\Gamma}  W'''(\phi_0)|\phi_0'|^2(z\phi_0' D_1+v_2) Z^2dsdz. 
\end{equation}
Using the relation \eqref{identities-K2-integration}  implies 
\begin{equation}\label{integration-K3-1}
    \int_{-\ell/\varep}^{\ell/\varep} \int_{\Gamma}  \mathrm K_3 dsdz =-2\int  \int_{-\ell/\varep}^{\ell/\varep} \int_{\Gamma}  |\phi''_0|^2 Z^2dsdz+  \int_{-\ell/\varep}^{\ell/\varep} \int_{\Gamma}  \phi''_0v_2 Z^2dsdz. 
\end{equation}
Now, using the definition of $v_2$ in equation \eqref{eq-v2-reduced}, showing in the appendix, and the even-odd parity of $\mathrm L_0v_2$-terms and $\phi_0''$, the second integral  above becomes 
\begin{equation}
     \int_{-\ell/\varep}^{\ell/\varep} \int_{\Gamma}  \phi''_0v_2 Z^2dsdz = 2\int  \int_{-\ell/\varep}^{\ell/\varep} \int_{\Gamma}  |\phi''_0|^2 Z^2dsdz. 
\end{equation}
Therefore the integration of the $\mathrm K_3$ is zero by \eqref{integration-K3-1}. The claim \eqref{claim-K3-integration-zero} is proved. Lastly, we consider the integral of $\mathrm K_2=\mathrm I_2|\phi_0'|^2 Z$. In fact, from the definition of $\mathrm I_2$ and using the relation \eqref{identities-K2-integration} implies 
\begin{equation}\label{K2-integration}
\begin{aligned} 
    \int_{-\ell/\varep}^{\ell/\varep} \int_{\Gamma}  \mathrm K_2 dsdz 
    & = -  \int_{-\ell/\varep}^{\ell/\varep}  |\phi_0''|^2 dz\|{H}Z\|_{L^2(\Gamma)}^2.
    \end{aligned} 
\end{equation}
Combining the equations \eqref{claim-K3-integration-zero}, \ref{K2-integration} with inequality \eqref{kernel-est-remainder-1} completes the proof of the Lemma. 
\end{proof}


\subsection{Orthogonal and crossing estimates} \label{sec-linear-coercivity-orthogonal}
 In this section, we demonstrate the coercivity of the linearized operator in the space perpendicular to $\varphi$ in the sense of \eqref{orthogonal-decomposition-Z}. This is based on the following coercivity of the Allen-Cahn operator, see \cite{chen1994spectrum} or Lemma $6.1$ in \cite{fei2021phase}.


\begin{lemma}\label{lem-coer-rL-Allen-Cahn}
    Suppose $w\in H^2$ is a given function satisfying the meandering orthogonal condition in \eqref{orthogonal-decomposition-Z},  then there exists a positive uniform constant $C$, independent of $\varep$, such that 
    \begin{equation}
        \varep^4 \|\mathrm L_\varep w\|_{L^2}^2 \geq C\|w\|_{\Htwoin}^2. 
    \end{equation} 
\end{lemma}
\begin{prop}\label{prop-est-perpendicular-terms}
 Let $w\in H^2$ be a given function satisfying the meandering orthogonal condition in \eqref{orthogonal-decomposition-Z},  then there exists a positive uniform constant $C$, independent of $\varep$, such that 
    \begin{equation}
        \varep^4 \left<\mathbb L_\varep w, w\right>_{L^2} \geq C\|w\|_{\Htwoin}^2. 
    \end{equation}
\end{prop}
\begin{proof}
    In light of the form of the linearized operator in \eqref{form-bL}, it holds that: 
    \begin{equation}
        \varep^4 \left<\mathbb L_\varep w, w\right>_{L^2} \geq \varep^4\|\mathrm L_\varep w\|_{L^2}^2 -\varep \|\mathrm R_a\|_{L^\infty}\|w\|_{L^2}^2.
    \end{equation}
    For a given $k$-approximate solution, $|\mathrm R_a|$ is uniformly bounded. This  implies
    \begin{equation}\label{est-orthogonal-0}
        \varep^4 \left<\mathbb L_\varep w, w\right>_{L^2} \geq \varep^4\|\mathrm L_\varep w\|_{L^2}^2  -C\varep\|w\|_{L^2}^2.
    \end{equation} 
   The proposition follows from the coercivity of the Allen-Cahn operator in Lemma \ref{lem-coer-rL-Allen-Cahn} and taking $\varep$ sufficiently small. 
\end{proof}

Lastly, we address the cross term $\left<\varep^4 \mathbb L_\varep u^\parallel, {w}\right>_{L^2}$. This step will complete the analysis of the linearized operator. The main result of this section is presented in Proposition \ref{prop-est-cross-terms}.
Squaring and expanding the expression in \eqref{expansion-rL} yields:
    \begin{equation}\label{expansion-double-rL-0}
        (\varep^2 \mathrm L_\varep)^2   -\mathbb L_0= -\varep \mathrm L_0\left[\Big( {H}(s)  +\varep zH_2(s)\Big)\p_z\right]-\varep^2\mathrm L_0 \circ(W'''(\phi_0)u_2)+\varep^2 {H}^2(s)\p_z^2+\mathrm R^{\mathrm L}
    \end{equation} 
    where the remainder $\mathrm R^{\mathrm L}=\sum_{k=1}^4\mathrm R^{\mathrm L}_k$ is given by 
    \begin{equation}\label{def-rRs}
    \begin{aligned} 
        \mathrm R^{\mathrm L}_1 &:= \varep \left[\varep^2 \Delta_\Gamma, {H}(s)\p_z\right]\\
        \mathrm R^{\mathrm L}_2&= \varep^2 \left[\varep^2 \Delta_\Gamma, zH_2(s)\p_z+W'''(\phi_0)u_2 +\varep\tilde{\mathrm D}_z\right]; \\
        \mathrm R^{\mathrm L}_3 &:=-\varep {H}(s)\p_z \mathrm L_0-\varep^2 \left( zH_2(s)\p_z+W'''(\phi_0)u_2\right) \mathrm L_0; \\
        \mathrm R^{\mathrm L}_4 &:=-\varep^3\left[\mathrm L_0, \tilde{\mathrm D}_z\right]+\varep^3 \left[{H}(s)\p_z, zH_2(s)\p_z +W'''(\phi_0)u_2+\varep \tilde{\mathrm D}_z\right] \\
        &\qquad +  \varep^4 \Big(zH_2(s)\p_z+W'''(\phi_0)u_2+\varep \tilde{\mathrm D}_z \Big)^2.
        \end{aligned} 
    \end{equation}
    Here, we have used the bracket notation $[a,b]=ab-ba$.   
\begin{prop}\label{prop-est-cross-terms}
    Let  $\varphi=\varphi(z,s)$ be as introduced in \eqref{def-varphi}.    For any given positive constant $\delta_*>0$, and for any  $u^\parallel=\varphi(z,s)Z(s)$ and ${w}$ satisfying the orthogonal condition in \eqref{orthogonal-decomposition-Z}, the following inequality holds: 
    \begin{equation}
        \varep^4\left<\mathbb L_\varep u^\parallel, {w}\right>_{L^2} \leq \delta_*\varep^4\|Z\|_{H^2(\Gamma)}^2 +C\varep^4 \|Z\|_{L^2(\Gamma)}^2+\delta_*\|{w}\|_{\Htwoin}^2. 
    \end{equation}
    Here, the constant $C$ is independent of $\varep$, but may depend on $\delta_*$. 
\end{prop}
The proof is direct using the expansion \eqref{expansion-double-rL-0}, the orthogonal condition \eqref{orthogonal-decomposition-Z} to cancel out the leading order,  and integration by parts to control higher order terms. See also Proposition $5.1$ in \cite{fei2021phase}.


\appendix
\section{}\label{Appendix-A}
In this appendix, we list some technical lemmas used in the article. 

\subsection{Relations between norms under Euclidean and local coordinates}
We quote from Lemma 6.2 in \cite{fei2021phase} some coercivity estimates of the Laplacian and inner-norm under local coordinates.  
\begin{lemma}\label{coer-lapace-ubot}
    For any $u \in H^2$, there exists some universal constant $C$ such that  
    \begin{equation}
        \|\varep^2\Delta u\|_{L^2}^2\geq \frac{1}{4} \left(\|\varep^2\Delta_\Gamma u\|_{L^2(\Gamma^\ell)}^2+ \|\p_z^2 u\|_{L^2(\Gamma^\ell)}^2\right)-C\varep^4(\|\nabla u\|_{L^2}^2 +\|u\|_{L^2}^2). 
    \end{equation}
    Moreover, the inner norm admits a better estimates as:  
    \begin{equation}
        \|{w}\|_{\Htwoin}^2\geq C\|\varep^2 \Delta_\Gamma {w}\|_{L^2(\Gamma^\ell)}^2+C\|\varep^2 \Delta_\Gamma^0 {w}\|_{L^2(\Gamma^\ell)}^2 +C\|\p_z^2 u\|_{L^2(\Gamma^\ell)}^2+C\|\p_z u\|_{L^2(\Gamma^\ell)}^2. 
    \end{equation}
\end{lemma}
\begin{proof}
    The first estimate can be proved via change of variable, for instance see Lemma 6.2 in \cite{fei2021phase}. The second estimate is a direct corollary of the first estimate from the definition of the norm in $\Htwoin$. 
\end{proof}

\subsection{Estimates in the local region $\Gamma^\ell$}  
We first point out a norm comparison  of localized functions itself and those functions  with polynomial multipliers $z^k$, where $z=\frac{d}{\varep}$ is the scaled normal variable.   
\begin{lemma}\label{lem-localized-function-control} 
    Let $k\geq  0, \nu>0$ be given constants.   There exists a positive constant $C$, depending only on $k$ and $\nu$, such that  for any function $u$ supported and localized near $\Gamma^\ell$ with decaying exponent $\nu>0$, the following holds: 
    \begin{equation}
        \int_\Omega |z|^k |u| dx\leq C\int_\Omega |u| dx. 
    \end{equation}
\end{lemma}
\begin{proof}
   Since $u$ is localized, there exists a positive constant $C_0>0$ such that $|u(x(z,s))|\leq C_0e^{-\nu |z|}$ in the local region $\Gamma^\ell$. Without loss of generality, we can assume $C_0 = 1$; otherwise, we consider $v = u/C_0$.

Next, assuming $u \neq 0$, we define a nonzero finite constant $M$ as:
    \begin{equation}
        \int_\Omega |u|dx:=M\neq 0. 
    \end{equation} 
    Let $R > 0$ be a constant depending on $\nu, M$, and $\Omega$, such that:
    \begin{equation}
        |z|^k e^{-\nu z}\leq \frac{M}{2|\Omega|}, \qquad \hbox{for}\qquad  |z|\geq R. 
    \end{equation} 
We define the exterior domain $\Omega_R^{\text{out}} := \{|z| \geq R\} \cap \Gamma_\ell$. For u supported and localized in $\Gamma_\ell$ with a decaying exponent $\nu,$ it holds that:
    \begin{equation}\label{localized-function-est-exterior}
        \int_{\Omega_{out}^R} |z|^k |u| dx \leq \frac{M}{2}. 
    \end{equation}
    For the interior domain $\Omega_R^{\text{in}} := \Omega \setminus \Omega_R^{\text{out}}$, we have the estimate:
    \begin{equation}\label{localized-function-est-interior}
        \int_{\Omega_{in}^R} |z|^k |u| dx \leq R^k \int_\Omega |u| dx \leq R^k M. 
    \end{equation}
    Summing these estimates for the interior and exterior domains, \eqref{localized-function-est-exterior} and \eqref{localized-function-est-interior}, gives:  
    \begin{equation}
        \int_\Omega |z|^k |u| dx\leq \left(\frac12 +R^k\right)M. 
    \end{equation}
    The Lemma follows from the definition of $M$. 
\end{proof} 
Below we characterize the behavior of the function $u$, which is localized near the surface $\Gamma$ and decays rapidly away from it. The norms of $u$ are controlled by the corresponding norms of the $s$-dependent function, with additional dependence on the small parameter $\varepsilon$ and the curvature of the surface $\Gamma$.
\begin{lemma}\label{lem-auxilliary-operator-Gamma-Gamma0}
   Suppose $\Gamma$ is smooth.  Let $\phi=\phi(z)$ be a nonzero smooth function on $\mathbb R$ which decays exponentially fast to zero as $|z|\to \infty$,  then  $u:=\frac{1}{m\sqrt\varep}\phi(z)Z(s)\chi(\varep z/\ell)$ where $m:=\|\phi\|_{L^2(\mathbb R)}$ is localized near $\Gamma^\ell$ and there exists a universal constant $C$ depending only on system parameters and $\Gamma$ such that  
    \begin{equation}
    \begin{aligned} 
        &\|u\|_{L^2}  \leq C \|Z\|_{L^2(\Gamma)}; \\
        &\|\nabla^k_\Gamma u\|_{L^2}^2 + \varep \int_{-\ell/\varep}^{\ell/\varep}\int_\Gamma  |\nabla^k_\Gamma u|^2 dsdz \leq C\|Z\|_{H^k(\Gamma)}^2; \\
        &\|(\Delta_\Gamma-\Delta_\Gamma^0)u\|_{L^2}\leq C\varep  \|h\|_{C^1}\|Z\|_{H^1(\Gamma)}; \\
        &\|\Delta _\Gamma u\|_{L^2} \geq C\|Z\|_{H^2(\Gamma)} -C_1\|Z\|_{L^2(\Gamma)}.
        \end{aligned} 
    \end{equation}
\end{lemma}
\begin{proof}
    The first and second inequality follows directly from the Jacobian expansion, and change of variable from the usual Euclidean coordinates to local coordinates. The third inequality is a direct corollary of the expansion of the Laplace-Beltrami operator. 

    The Laplace-Beltrami operator, $\Delta_\Gamma$, around the $r$-level surface $\Gamma^{\ell,r}$ is defined in \eqref{def-Laplace-Beltrami-Gamma}. The first fundamental form has  expansion as in \eqref{1st-fundamental form-expansion}, which implies
    \begin{equation}
        \sum_{k=0}^1|\p_{s_l}^{k}(g_{ij}(r,s)-g_{ij}(0,s))|\leq C\varep \|h\|_{C^1}\left(|z|+\varep|z|^2\right)
    \end{equation}
    Similar estimate holds for $g$ and $g^{ij}$. There the difference of the Laplace-Beltrami operator is given by 
    \begin{equation}
        |(\Delta_\Gamma-\Delta_\Gamma^0)u|\leq \varep \|h\|_{C^1} (|z|+\varep|z|^2)\left(|\nabla_\Gamma u| +|u|\right). 
    \end{equation}
   Note that $u$ is localized and supported in $\Gamma^\ell$. The inequality above, together with  Lemma \ref{lem-localized-function-control}, implies 
   \begin{equation}
       \|(\Delta_\Gamma-\Delta_\Gamma^0)u\|_{L^2}\leq \varep \|h\|_{C^1} \left(\|u\|+\|\nabla_\Gamma u\|_{L^2}\right)
   \end{equation}
   The third inequality then follows from the second inequality. 

   We now prove the last inequality. From the third inequality, it suffices to show that there exists positive constants $C, C_1$ such that 
   \begin{equation}
       \|\Delta_\Gamma^0u\|_{L^2}\geq C\|Z\|_{H^2(\Gamma)}-C_1\|Z\|_{L^2(\Gamma)}. 
   \end{equation}
   Note that $\Delta_\Gamma^0 u = \frac{1}{m\sqrt\varep} \phi(z)\Delta_\Gamma^0Z(s)\chi(\varep z/\ell)$, where $\phi$ decays exponentially away from the interface. Therefore, using the lower bound of the Jacobian  in \eqref{lower-bdd-Jacobian} implies  
   \begin{equation}
       \|\Delta_\Gamma^0 u\|_{L^2} \geq C\|\Delta_\Gamma^0Z\|_{L^2(\Gamma)}. 
   \end{equation}
    
\end{proof}

\section{}\label{Appendix-B}
In the appendix, we solve the equations in different orders which gives the form of $k$-approximate solutions. The zeroth order equation is given in \eqref{eq-u-v-0}, first order in \eqref{eq-uv1}, 2nd order in \eqref{eq-uv2}, and any $k(\geq 3)$-order in \eqref{system-rE_j}.
\subsection{Zeroth and first order} 
In terms of $\mrL_0$, see \eqref{def-rL}, the zeroth order equations take the form 
\begin{equation}\label{eq-u-v-0}
  (\mathrm{Eq}_0)  \left\{ \begin{aligned}
        &\p_z^2u_0 -W'(u_0)=0, \\
        &  \mrL_0v_0 = -\sigma_0. 
    \end{aligned}\right. 
\end{equation}
The right hand side $\sigma_0$ is even and  not perpendicular to $\phi_0'$, thus $v_0$ is solvable only when $\sigma_0=0$. The first order equation takes the form  
\begin{equation}\label{eq-uv1}
   (\mathrm{Eq}_1) \left\{\begin{aligned} 
        \mrL_0 u_1 &= -v_0+\Delta d_0 \phi_0'; \\
        \mrL_0 v_1 &= -W'''(\phi_0) u_1 v_0 -\sigma_1+\left(\Delta d_0 +2\nabla d_0\cdot \nabla \right)\p_zv_0.   
    \end{aligned}\right. 
\end{equation} 
Here $\phi_0$ is the heteroclinic profile introduced in \eqref{eq-phi0}.
\begin{lemma}\label{lem-uv01} 
   Suppose that $\sigma_0=\sigma_1=0.$ Then for any $v^\parallel=l(x,t)\phi_0'(z)$ for some smooth function $l$, the following is a   solution to the zeroth and first  order equations \eqref{eq-u-v-0}:
   \begin{equation}
   \begin{aligned} 
       & u_0=\phi_0, \qquad v_0 =\phi_0'\Delta d_0; \\
       & u_1 = 0, \qquad v_1=-D_0z\phi_0'+v_1^\parallel. 
       \end{aligned}
  \end{equation}
\end{lemma}
Hereafter, we use $u^\parallel,v^\parallel$ denote functions align with $\phi_0'$ and $u^\bot, v^\bot$ to denote functions perpendicular to $\phi_0'$ on $\mathbb L(\mathbb R)$. 

\subsection{Second order.}
With $(u_0,v_0)$ given in Lemma \ref{lem-uv01}, and $\mrL_0$ defined in \eqref{def-L0} $(u_2,v_2)$ solves 
\begin{equation}\label{eq-uv2}
   (\mathrm{Eq}_2) \left\{\begin{aligned}
        \mrL_0 u_2 &=-v_1 +\Delta d_1\phi_0';\\ 
        \mrL_0v_2 &= -W'''(u_0)u_2\phi_0' \Delta d_0-\sigma_2 +\p_t d_0 \phi_0' +(\Delta d_1+2\nabla d_1\cdot \nabla )\Delta d_0 \phi_0''\\
        &\quad +(\Delta d_0+2\nabla d_0\cdot \nabla)\p_z v_1  +\Delta^2 d_0\phi_0'-E_0d_0{\phi_0'}. 
    \end{aligned}\right. 
\end{equation}
We introduce the zeroth order geometric differential operator $\mathrm G_0$ given by 
\begin{equation}
    \mathrm G_0[d_0,\sigma_2] :=  \p_td_0+\Delta^2 d_0 -\left(\Delta d_0+\nabla d_0\cdot \nabla\right)D_0 -\frac{2\sigma_2}{m_1^2}.  
\end{equation}
\begin{lemma}\label{lem-u2v2}
Suppose $\Gamma_0$ is  a smooth solution to the volume preserving Willmore flow, \eqref{def-dynamics-Gamma0}-\eqref{Area-constraint-leading-order}; and $E_0$ is given by 
    \begin{equation}
        E_0 = \left\{\begin{aligned}&  \frac{\mathrm G_0[d_0,\sigma_2]}{d_0}, \qquad \hbox{on $\Gamma^\ell_0\setminus \Gamma_0$}; \\
       &   \nabla \mathrm G_0[d_0,\sigma_2]\cdot \nabla d_0, \qquad \hbox{on $\Gamma_0$}. 
        \end{aligned}\right.
    \end{equation}
    Then if $v_1^\parallel  = \Delta d_1 \phi_0'$,   the second order system \eqref{eq-uv2} has solution $(u_2,v_2)$ on $\mathbb R\times \Omega\times (0,T)$, where $u_2$ is even and given by: 
    \begin{equation}
        u_2:= D_0\mathrm L_0^{-1}(z\phi_0'). 
    \end{equation}
    and $v_2$ solves 
    \begin{equation}\label{eq-v2-reduced}
    \begin{aligned} 
        \mathrm L_0 v_2 &= 2D_1\phi_0'' -W'''(\phi_0) u_2 \phi_0' \Delta d_0 +\left(\Delta d_0+2\nabla d_0\cdot \nabla\right) \p_z v_1^\bot +\mathrm G[d_0,\sigma_2]\phi_0'\\
        &\quad +\left(\Delta d_0+\nabla d_0\cdot \nabla \right)D_0\phi_0'+\frac{\sigma_2}{m_1^2}(2\phi_0'-m_1^2). 
        \end{aligned} 
    \end{equation}
    The terms on the right hand side are even with respect to the $z$-variable, except the first term: $-2D_1\phi_0''$. 
\end{lemma}
\begin{proof} The solvability of $u_2$ gives the form of $v_1^\parallel$, and then one can solve the $u_2$ equation to get its form.  As before, we care about the projection of the right hand side of the $v_2$ equation to $\phi_0'.$   To this goal, from the second identity in Lemma \ref{lem-L0} and $\mathrm L_0u_2=-v_1^\bot$  we deduce 
    \begin{equation} 
    \begin{aligned} 
        -\int_{\mathbb R}W'''(u_0)|\phi_0'|^2 u_2 dz& = - \int_{\mathbb R}\phi_0'' v_1^\bot dz  = \int_{\mathbb R}\p_z  v_1^\bot \phi_0'dz. 
        \end{aligned} 
    \end{equation}
   Denoting the right hand side of the $v_2$ equation by $\mathrm{RHS}_2$, using the identity above and definition of $v_1$ yields 
   \begin{equation}
   \begin{aligned} 
       \int_{\mathbb R} \mathrm{RHS}_2 \phi_0' dz 
       &=\mathrm G_0[d_0,\sigma_2] m_1^2 - E_0d_0m_1^2. 
       \end{aligned} 
   \end{equation}
   Setting the right hand side to be zero gives the form of $E_0$ outside $\Gamma_0$ where $d_0\neq 0$. To preserve the continuity the form of $E_0$ follows from L'Hospital's rule and $|\nabla d_0|=1$. 
\end{proof}


\subsection{$k(\geq 3)$-th order} 
We adopt the following expansion introduced  in \cite{fei2021phase}, for $u_a=\sum_{k\geq 0}\varep^ku_k$
\begin{equation}
    \begin{aligned}
        W'(u_a) &= W'(u_0) +W''(u_0) \sum_{i\geq 1} \varep^i u_i +\sum_{i\geq 1} \varep^i \mathcal W^{(1)}_{i-1}(u_0, \cdots , u_{i-1}); \\
        W''(u_a) &= W''(u_0) +W'''(u_0) \sum_{i\geq 1} \varep^i u_i +\sum_{i\geq 1} \varep^i \mathcal W^{(2)}_{i-1}(u_0, \cdots , u_{i-1}) 
    \end{aligned}
\end{equation}
 where $\mathcal W_i^{(k)}(k=1,2)$ are polynomials of $(i+1)$-variables. They vanish at the origin and $\mathcal W_0^{(k)}=0$ for both $k=1,2$. Moreover for $u_1=0$, we also have
 \begin{equation}
     \mathcal W_1^{(k)}=\mathcal W_2^{(k)}=0, \quad \hbox{for $k=1,2$.}
 \end{equation}
 With these notations, the $\varep^{k+2}$-order equation becomes 
 \begin{equation}\label{system-rE_j}
    (\mathrm{Eq}_{k+2}) \left\{ \begin{aligned}
         \mathrm L_0 u_{k+2} &= -\mathcal W_{k+1}^{(1)} - v_{k+1} +\sum_{0\leq i\leq k+1} (\Delta d_i+2\nabla d_i \cdot\nabla) \p_z u_{k+1-i} +\Delta u_k; \\
         \mathrm L_0 v_{k+2} &= -\sigma_{k+2} -\sum_{1\leq i\leq k+2} \left(W'''(\phi_0)u_i+\mathcal W_{i-1}^{(2)}\right) v_{k+2-i} +\Delta v_k\\
         &\quad +\sum_{0\leq i\leq k+1}(\Delta d_i+2\nabla d_i\cdot \nabla)\p_z v_{k+1-i}  +\sum_{0\leq i\leq k} \p_t d_i \p_z u_{k-i} +\p_t u_{k-1} \\
         &\quad  -\sum_{0\leq i\leq k}E_i d_{k-i}{\phi_0'}+E_{k-1}z{\phi_0'}.
     \end{aligned}\right.
 \end{equation}
 The idea of solving this linear system of $(u_{k+2},v_{k+2})$ on $\mathbb R\times \Omega \times (0,T)$ is by mathematical induction and choosing appropriate $(d_j, E_j)_{j=0}^k$ and $(v_j^\parallel)_{j=1}^{k+1}$ so that the right-hand sides of the system satisfies the compatibility condition \eqref{cond-compatibility}.  It's useful to introduce the $k$-th order geometric operator as 
\begin{equation}
    \mathrm G_k[d_k,\sigma_{k+2}]:= \p_t d_k+\Delta^2 d_k -\sum_{l=0,k} \left(\nabla D_l\cdot \nabla d_{k-l} +D_l\Delta d_{k-l}\right)-\frac{2\sigma_{k+2}}{m_1^2}. 
\end{equation} 
As a convention of notation we accept  that $\mathscr A_{j}$ depends on terms of order up to $j$ and $\sigma_{j+1}$, the term $\mathscr A_{j}^*$ depends on terms of order up to $j$ but might also depends on $\sigma_{j+1},\sigma_{j+2}$. Moreover, these terms decays exponentially fast to  constants as $|z|\to\infty.$ Similarly $\mathscr Q_j=\mathscr Q_j(x,t)$, independent of the $z$-variable, denotes terms depending on geometric quantity defined in $\Gamma^\ell$ subject to order $j$ and $\sigma_{j+1}$, and $\mathscr Q_j^*$ might also depend on $\sigma_{j+2}$.  For simplicity of notation,  we don't track their form and they might change line by line. With these notations, we introduce the assumption with index $k$: 
\begin{equation} \label{Statements-bAk} 
    (\mathbf A_k) 
    \left\{\begin{aligned}
        &\hbox{(1) $v_j^\parallel = \Delta d_j \phi_0'+\mathscr A^*_{j-2}$ for $1\leq j\leq k$};\\
        &  \hbox{(2) $(d_j,E_j)$  for $0\leq j\leq k-1$ satisfies the relation}  \\
       &\qquad E_j = \left\{\begin{aligned}
           &\frac{1}{d_0} \left(\mathrm G_j[d_j,\sigma_{j+2}]-E_0d_j+\mathscr Q_{j-1}^*\right) , \quad \hbox{on $\Gamma_0^\ell\setminus \Gamma_0$}; \\
           &\nabla \left(\mathrm G_j[d_j,\sigma_{j+2}]-E_0d_j+\mathscr Q_{j-1}^*\right)\cdot \nabla d_0, \quad \hbox{on $\Gamma_0$}. 
       \end{aligned}
       \right.
    \end{aligned}\right. 
\end{equation}
Here we accept the convention that  $\mathscr A_{-1}^*\equiv 0$. We shall argue by induction and put the following statement with index $k\geq 1$: let $(u_j,v_j)$ be solutions to equations $(\mathrm{Eq}_j)$ for $0\leq j\leq k+1$,  
\begin{equation} \label{Statements-bHk} 
    (\mathbf H_k) 
    \left\{\begin{aligned}
        & \hbox{(1) $u_j=\mathscr A_{j-2}$  
        for $2\leq j\leq k+1$;} \\
        &\hbox{(2) $v_j^\bot=  -D_{j-1}z\phi_0'+\mathscr A^*_{j-2}$ for $2\leq j\leq k+1$;} \\
        &\hbox{(3) $(u_j, v_j^\bot)$ decays exponentially fast to a constant as $|z|\to \infty$ for $2\leq j\leq k+1$;} \\
    \end{aligned}\right. 
\end{equation}
For function $u_2$  given in Lemma \ref{lem-u2v2}, the first statement in $(\mathbf H_1)$ is valid. With $u_2$ given above, solving the $v_2$ equation in \eqref{eq-v2-reduced} with the aid of the third identity in Lemma \ref{lem-L0} and form of $v_1^\bot$  in Lemma \ref{lem-uv01} implies 
\begin{equation}
    v_2^\bot =  - D_1 z\phi_0' +\mathscr A_0^*. 
\end{equation}
The second statement in $(\mathbf H_1)$ is correct.
\begin{lemma}
    For some  $(\mathscr A_{0}^*, \mathscr Q_0^*)$  such that $(\mathbf A_2)$ holds,  there are some $(u_3, v_3^\bot)$   fulfilling the corresponding statements in $(\mathbf H_2)$, \eqref{Statements-bHk},  such that $(u_3, v_3)$,  with $v_3=v_3^\bot+v_3^\parallel$ where $v_3^\parallel=l_3(x,t)\phi_0'(z)$ for any smooth function $l_3(x,t)$, solves  the system $(\mathrm{Eq}_3)$ \eqref{system-rE_j} with a given $\sigma_3$ on $\mathbb R\times \Omega\times (0,T)$.
\end{lemma}
   \begin{proof}
Putting $k=1$ in equations $(\mathrm{Eq}_k)$ and using $u_1=0$ yields 
\begin{equation*}
   \left\{ \begin{aligned}
        \mathrm L_0 u_3 & = -\mathcal W_2^{(1)} -v_2  +\phi_0' \Delta d_0 +(\Delta d_0+2\nabla d_0\cdot \nabla)\p_z u_2; \\
        \mathrm L_0 v_3 & = -\sigma_3 +2D_2\phi_0'' -\left(W'''(\phi_0)u_2 +\mathcal W_1^{(2)}\right)v_1 -\left(W'''(\phi_0)u_3 +\mathcal W_2^{(2)}\right)v_0 \\
        &\quad  +(\Delta d_0 +2\nabla d_0 \cdot \nabla )\p_z v_2^\bot +(\Delta d_2 +\nabla d_2\cdot \nabla) \p_z v_0^\bot+(\Delta d_1 +2\nabla d_1 \cdot \nabla ) \p_z v_1^\bot  \\
        &\quad  +\p_t d_1 \phi_0' +\Delta v_1-(E_0d_1+E_1d_0){\phi_0'}+E_0z{\phi_0'}. 
    \end{aligned}\right. 
\end{equation*}
Deriving similarly as in \cite{fei2021phase} for some $v_2^\parallel$ in the form in $(\mathbf A_2)$, the $u_3$ equation is solvable and in the form of 
\begin{equation}
\begin{aligned} 
    \mathrm L_0 u_3 &=D_1z\phi_0'  +\mathscr A_0^*. 
    \end{aligned} 
\end{equation}
And there exists $(d_1,E_1)$ as in $(\mathbf A_2)$ such that the compatibility condition holds for the right hand side  of $v_3$ equation, and $v_3$ is solvable. Particularly in our case,  the $v_3$-equation reduces to the form 
\begin{equation}
    \mathrm L_0v_3 = \frac{\sigma_3}{m_1^2} \left(2\phi_0'-m_1^2\right) +2D_2 \phi_0'' +\mathscr A_1^*. 
\end{equation}
Solving the equation determines the form of $v_3^\bot$ and the statement $(2)$ in $(\mathbf H_2)$ follows. 
\end{proof}

\begin{lemma}
    Let $k\geq 3$ be any integer and $u_k,v_k$ has the form given in $(\mathbf H_{k})$, then it holds that 
    \begin{equation}\label{eq-uk-vk-W}
    \begin{aligned} 
        -\int_{\mathbb R} W'''(\phi_0)  |\phi_0'|^2 u_{k} dz  &= -\frac{D_{k-2}}{2}m_1^2 +\mathscr Q_{k-3}^*;\\
        \int_{\mathbb R} \p_z v_k \phi_0' dz &= -\frac{D_{k-1}}{2}m_1^2+\mathscr Q_{k-2}^*.
        \end{aligned} 
    \end{equation} 
\end{lemma}

\begin{prop}
   Let $k\geq 2$ be an integer. Suppose  $(\mathbf H_k)$ is true for $(u_j,v_j)_{j=1}^{k+1}$, then for some  $(v_{k+1}^\parallel, d_{k-1}, E_{k-1})$ has the form in $(\mathbf A_{k+1})$, the system \eqref{system-rE_j} has a solution  $(u_{k+2}, v_{k+2}^\bot)$  fulfilling the statements in $(\mathbf H_{k+1})$. 
\end{prop}
\begin{proof}
    In view of the first equation in $(\mathrm{Eq}_k)$, using the definition of $\mathscr A_k^*$ and $u_{k+1}=\mathscr A_{k-1}^*$ for some $\mathscr A_{k-1}^*$ yields 
    \begin{equation}
        \mathrm L_0 u_{k+2} = -v_{k+1}+\Delta d_{k+1}\phi_0' +\mathscr A_{k-1}^*. 
    \end{equation}
    The solvablity of  $u_{k+2}$ determines the form of $v_{k+1}^\parallel$, which with the aid of the $v_{k+1}^\bot$-form stated in $(\mathbf H_k)$ reduces the equation above to 
    \begin{equation}
        \mathrm L_0u_{k+2} =D_k z\phi_0'+\mathscr A_{k-1}^*. 
    \end{equation} 
    This can also be written as $\mathrm L_0 u_{k+2} = \mathscr A_{k}$, the first statement in $(\mathbf H_{k+1})$ follows. 

    Now we turn to the equation of $v_{k+2}$. Note that $\mathcal W_{k+1}^{(2)}=\mathcal W_{k+1}^{(2)}(u_0,\cdots,u_{k+1})=\mathscr A_{k-1}^*$ from the first statement of $(\mathbf H_k)$. Using the form of $v_{j}^\parallel$ for $j\leq k+1$ and definition of $D_{k+1}$ implies  
    \begin{equation*}
        \begin{aligned}
            \mathrm L_0 v_{k+2} = & -\sigma_{k+2} -2D_{k+1}\phi_0'' - W'''(\phi_0) u_{k+2} v_0 -W'''(\phi_0) u_{2}v_{k} + (\Delta d_k+2\nabla d_k\cdot \nabla)\p_z v_1^\bot \\
            & + (\Delta d_0+2\nabla d_0\cdot \nabla)\p_zv_{k+1}^\bot +\p_t d_k \phi_0' +\Delta v_k -E_kd_0\phi_0'-E_0 d_{k}\phi_0'+\mathscr A^*_{k-1}. 
        \end{aligned}
    \end{equation*}
   Note that $v_k=\Delta d_k\phi_0'+\mathscr A_{k-1}^*$ we have
    \begin{equation*}
        \begin{aligned}
            \mathrm L_0 v_{k+2} = & -\sigma_{k+2} -2D_{k+1}\phi_0'' - W'''(\phi_0) u_{k+2} \phi_0' \Delta d_0 -W'''(\phi_0) u_{2}\phi_0' \Delta d_k + (\Delta d_k+2\nabla d_k\cdot \nabla)\p_z v_1^\bot \\
            & + (\Delta d_0+2\nabla d_0\cdot \nabla)\p_zv_{k+1}^\bot+\p_t d_k \phi_0' +\phi_0'\Delta^2 d_k\\
            &-E_kd_0{\phi_0'}-E_0 d_{k}{\phi_0'}+\mathscr A^*_{k-1}. 
        \end{aligned}
    \end{equation*}
    The compatibility condition implies the right hand side is perpendicular to $\phi_0'$, which is true for the first two terms on the right hand side. Moreover, using the identities in \eqref{eq-uk-vk-W}  implies  the existence of $(d_k, E_k)$ as  in $(\mathbf A_{k+1})$ so that the right hand side of $v_{k+2}$-equation satisfies the compatibility condition. 
    On the other hand from the form of $v_{k+1}^\bot$ in the statement $(2)$ of $(\mathbf H_k)$,   we can also rewrite the $v_{k+2}$-equation as 
    \begin{equation}
        \mathrm L_0 v_{k+2} = -2D_{k+1} \phi_0'' +\mathscr A_{k}^*. 
    \end{equation}
    The form of $v_{k+2}^\bot$ is determined. This completes the proof. 
 \end{proof}

\subsection{Mass condition and background state} 
Inductively, the mass condition determines the surface enclosed volume at order $k$, $\mathcal V_k$, then the resulting volume constraint and the dynamics of $d_k$ determines the Lagrange multiplier $\sigma_{k+2}$.  
\begin{lemma}
Let $k\geq 0$ and $0\leq j\leq k$, $\Gamma_j$ be the surface determined by the dynamics of $d_j$,  $\Gamma^{(k)}_a$ be the accumulated surface determined by $d=\sum_{j=0}^k\varep^jd_j$.  Suppose the surface volume $\{|\mathcal V_j|\}_{j=0}^{k-1}$ is finite and independent of $\varep$ and $|\Gamma^{(k)}|$ is given, then $|\mathcal V_k|$ is determined as a function of $\{|\mathcal V_j|\}_{j=0}^{k-1}$ and $|\mathcal V^{(k)}|$. 
\end{lemma}
\begin{proof}
   This is direct. 
\end{proof}

\begin{lemma}\label{lem-relation-rAk-sigmak} Let $k\geq 1$. Suppose $|\mathcal V_0|=\frac12(|\Omega|-M_0)$ and $(u_j)_{j=0}^{k+1}$ solves the system $(\mathrm E_j)$ and fulfills corresponding statements in $(\mathbf H_k)$. Suppose there exist $(\sigma_{j+2}, \mathcal V_j)_{j=0}^{k-1}$ such that for any $0\leq l\leq k$  the gluing function $u^{(l)}_a=\sum_{j=0}^{l}u_j^g,$  satisfies the mass condition up to order $\varep^{l}$, then $\mathrm V_k$ is determined so that the gluing function $u^{(k+1)}_a=u^{(k)}_a+u_{k+1}^g$ around $\Gamma^{(k+1)}_a$ satisfies the mass condition  up to order $\varep^{k}$. Moreover,
\begin{equation}
\mathcal V_k =\mathscr A_{k-1}^\infty.     
\end{equation}
\end{lemma}
\begin{proof}
  Since the gluing function $u_a^{(k+1)}=\sum_{j=0}^{k+1}\varep^ju_j^g$ satisfies the mass condition  up to order $\varep^{k+1}$, that is, there exists a positive constant $C$ such that 
\begin{equation}\label{Appendix-mass-cond-order-k+1}
    \begin{aligned}
        \left| \int_\Omega udx - M_0\right|\leq C\varep^{k+2}. 
    \end{aligned}
\end{equation}
Similarly as before we separate the domain to $\Omega^\pm$ and 
\begin{equation}
      \int_\Omega u_0^g dx- M_0 = |\Omega|-2\mathcal V_a +\varep \mathscr A_k^\infty+O(\varep^{k+2}).
    \end{equation}
Using the expansion of the surface enclosed volume, $\mathcal V_a=\sum_{j=0}^{k+1}\varep^j\mathcal V_j$, we have 
\begin{equation}
      \int_\Omega u_0^g dx-  M_0 =\mathscr A_{k}^\infty-2\varep^{k+1}  \mathcal V_{k+1}+O(\varep^{k+2}). 
    \end{equation}
Noting $u_1\equiv 0$. For $2\leq j\leq k$, one can derive similarly and 
\begin{equation}
    \varep^j \int_\Omega u_j^g dx =\mathscr A_{k}^\infty+O(\varep^{k+2}). 
\end{equation}
Summing the previous two identities implies 
\begin{equation}\label{Appendix-form-mass-order-k}
    \sum_{j=1}^k \varep^j \int_\Omega u_j^g dx =  -2\varep^{k+1}  \mathcal V_{k+1} +\mathscr A_{k}^\infty+O(\varep^{k+2}).
\end{equation}
The mass condition \eqref{Appendix-mass-cond-order-k+1} can now be interpreted as 
\begin{equation}\label{Appendix-Ak-dependence-0}
   2 \mathcal V_{k+1}=   \varep^{-(k+1)} \left(\int_\Omega \varep^j \sum_{j=0}^k u_j^gdx-\varep M_0+2\varep^{k+1}\mathcal V_{k+1}\right) +\int_\Omega u_{k+1}^g dx+O(\varep).  
\end{equation}
Since the glued function $\sum_{j=0}^ku_j^g$ satisfies the mass condition up to order $\varep^k$, the first term on the right hand side in the order of $\varep^0$, and hence $\mathscr A_k^\infty$ by \eqref{Appendix-Ak-dependence-0}. 
The second term on the right hand side can be handled by decomposing the domain and particularly   
\begin{equation}\label{mass-cond-u-k+1-1}
    \int_\Omega u_{k+1}^g dx = u_{k+1}^+ |\Omega^+| +u_{k+1}^-|\Omega_-| + O(\varep). 
\end{equation}
Here we recall that $u_{k+1}^\pm$ is a constant, which denotes the far field of $u_{k+1}$ as $z\to\pm \infty$. 
Moreover using the form of $u_{k+1}$ in $(\mathbf H_k)$ implies
\begin{equation}\label{mass-uj}
    \int_\Omega u_{k+1}^g dx =\mathscr A_{k}^\infty+  O(\varep). 
\end{equation}
Now combining this identity with \eqref{Appendix-Ak-dependence-0} implies the dependence  of $\mathcal V_k$. 
\end{proof}

\begin{cor}
    For $k\geq 0$, $\sigma_{k+2}$ is determined by the volume constraint of $\Gamma_a^{(k)}$ or $\Gamma_k$. Particularly, $\sigma_{k+2}$ depending on geometric quantities of $\Gamma_j$ up to order $k$, that is, 
       $\sigma_{k+2}=\mathscr A_k^\infty $. 
\end{cor}
\begin{proof} 
Note that the approximate surface $\Gamma_a^{(k)}$ is defined by \eqref{def-approx-surface} with $d_a(x,t)=\sum_{j=0}^k \varep^j d_j(x,t)$. Therefore 
\begin{equation}
    \frac{d}{dt} \mathrm V_a = \int_{\Gamma_a}  \mathbf n_a \cdot \p_t \mathbf X_a ds_a
\end{equation}
We consider terms at order $\varep^k$, and denote the $\varep^k$-order term on the left(right) hand side as $(\mathrm{LHS})_k((\mathrm{RHS})_k)$. Particularly, the $\varep^k$-term on the right hand side takes the form  
\begin{equation}
   (\mathrm{RHS})_k=  \int_{\Gamma_0}  \p_t d_k ds +\mathscr A_k^\infty(t).
\end{equation}
Here $\mathscr A_k^\infty(t)$ denotes dependence on  geometric quantities of $\{\Gamma_j\}$ up to order $j=k$. Particularly, it depends on $\Gamma_k$ linearly. Plugging the dynamics of $d_k$  on $\Gamma_0$ yields 
\begin{equation}
    (\mathrm{RHS})_k =   \frac{2\sigma_{k+2}}{m_1^2}|\Gamma_0| +\mathscr A_k^\infty(t). 
\end{equation}
The Lemma follows since $\mathrm V_a = \mathscr A_k^\infty.$
\end{proof}


 
\bibliographystyle{siam}
\bibliography{ref}

\end{document}